\documentclass{article}

\usepackage{amsmath}
\usepackage{amssymb}
\usepackage{amsthm}
\usepackage{mathtools}
\usepackage{nicematrix}
\usepackage{titlesec}
\usepackage{tikz-cd}
\usepackage{geometry}
\usepackage{amsmath, amssymb, graphics, setspace}
\usepackage{listings}

\newcommand{\mathsym}[1]{{}}
\newcommand{\unicode}[1]{{}}

\newtheorem{theorem}{Theorem}[section]
\newtheorem{proposition}[theorem]{Proposition}

\newtheorem*{theoremA}{Theorem A}
\newtheorem*{theoremB}{Theorem B}
\newtheorem*{theoremC}{Theorem C}
\newtheorem{lemma}[theorem]{Lemma}
\newtheorem{sublemma}[theorem]{Sublemma}
\newtheorem{corollary}[theorem]{Corollary}
\theoremstyle{definition}
\newtheorem{definition}[theorem]{Definition}

\newtheorem{remark}[theorem]{Remark}

\newtheorem{question}{Question}
\newtheorem{conjecture}[theorem]{Conjecture}

    \usepackage[T1]{fontenc}
    \usepackage{mathpazo}

    \usepackage{graphicx}
    \makeatletter
    \def\maxwidth{\ifdim\Gin@nat@width>\linewidth\linewidth
    \else\Gin@nat@width\fi}
    \makeatother
    \let\Oldincludegraphics\includegraphics
    \renewcommand{\includegraphics}[1]{\Oldincludegraphics[width=.8\maxwidth]{#1}}
    \usepackage{caption}
    \DeclareCaptionLabelFormat{nolabel}{}
    \usepackage{adjustbox} 
    \usepackage{xcolor} 
    \usepackage{enumerate} 
    \usepackage{geometry} 
    \usepackage{amsmath} 
    \usepackage{amssymb} 
    \usepackage{textcomp} 
    \AtBeginDocument{%
        \def\PYZsq{\textquotesingle}
    }
    \usepackage{upquote} 
    \usepackage{eurosym} 
    \usepackage[mathletters]{ucs} 
    \usepackage[utf8x]{inputenc} 
    \usepackage{fancyvrb} 
    \usepackage{grffile} 
    \usepackage{hyperref}
    \usepackage{longtable} 
    \usepackage{booktabs}  
    \usepackage[inline]{enumitem} 
    \usepackage[normalem]{ulem} 
    \usepackage{mathrsfs}

\title{On the geometry of the higher dimensional Voisin maps}
\author{Chenyu Bai}
\date{}

\begin{document}

\maketitle
\begin{abstract}
Voisin constructed self-rational maps of Calabi-Yau manifolds obtained as varieties of $r$-planes in cubic hypersurfaces of adequate dimension. This map has been thoroughly studied in the case $r=1$, which is the Beauville-Donagi case. In this paper, we compute the action of $\Psi$ on holomorphic forms for any $r$. For $r=2$, we compute the action of $\Psi$ on the Chow group of $0$-cycles,  and confirm  that it is as expected from the generalized Bloch conjecture.
\end{abstract}

\section{Introduction}
Let $X$ be a smooth projective variety of dimension $n$ defined over the field of complex numbers $\mathbb C$. Let $CH^k(X)_\mathbb Q$ denote the Chow group of $X$ with rational coefficients. In this work we will study the ring structure on $CH^*(X)_\mathbb Q$ for certain Calabi-Yau manifolds. This article is primarily motivated by a conjecture put forth by Voisin, which we will explore in detail:
\begin{conjecture}[Voisin]\label{ConjVoisinCY}
    Let $X$ be a strict Calabi-Yau manifold of dimension $n$. Let $C\subset CH_0(X)_\mathbb Q$ be the subgroup generated by the intersections of divisors and of Chern classes of $X$. Then the cycle class map 
    \[
    cl: CH_0(X)_\mathbb Q\to H^{2n}(X,\mathbb Q)
    \]
    is injective on $C$.
\end{conjecture}

\subsection{Voisin's examples of Calabi-Yau manifolds and Voisin maps}\label{SectionVoisinExample}
\begin{definition}
A \emph{strict Calabi-Yau manifold} is a simply connected compact Kähler manifold  $X$ of dimension at least $3$, with trivial canonical bundle and such that $H^0(X,\Omega_X^k)=0$ for $0<k<{\rm dim}\,X$.
\end{definition}

\begin{remark}
According to Kodaira's embedding theorem, such $X$ is automatically projective.
\end{remark}

Strict Calabi-Yau manifolds are pivotal in the study of algebraic varieties with trivial canonical bundles, known as $K$-trivial varieties. This class also includes hyper-Kähler manifolds, which are simply connected manifolds and possess a non-degenerate holomorphic $2$-form, for which Conjecture~\ref{ConjVoisinCY} is also expected to be true when they are projective (this is the Beauville-Voisin conjecture~\cite{BeauvilleSplitting, 0CycleHK}), and complex tori, which are defined as quotients of $\mathbb{C}^n$ by a lattice $\Gamma$, for which Conjecture~\ref{ConjVoisinCY} is definitely wrong, when they are projective.

Now we focus on specific families of $K$-trivial varieties as constructed in~\cite{KCorr}, generalizing the Beauville-Donagi construction~\cite{BeauvilleDonagi}. Let $Y \subset \mathbb{P}^n$ be a smooth cubic hypersurface of dimension $n-1$, and let $r \geq 0$ denote a nonnegative integer. Define $X=F_r(Y)$ as the Hilbert scheme that parametrizes the $r$-dimensional linear subspaces in $Y$. As proved in~\cite[(4.41)]{KCorr}, for $n+1=\binom{r+3}{2}$ and a general $Y$, the variety $X$ is a $K$-trivial variety of dimension $N = (r+1)(n-r)-\binom{r+3}{3}$. Specifically, when $r=0$, $X$ is an elliptic curve; and as established in~\cite{BeauvilleDonagi}, $X$ is a hyper-Kähler manifold for $r=1$. We have the following

\begin{lemma}\label{LmmStrictCY}
For $r \geq 2$ and $n+1 = \binom{r+3}{2}$, the constructed variety $X=F_r(Y)$ is a strict Calabi-Yau manifold.
\end{lemma}
\begin{proof}
Reference \cite[Proposition 3.1(a)]{DebarreManivel} indicates that, generally, if $n \geq \frac{2}{r+1}\binom{r+3}{r} + r + 1$, then $F_r(Y)$ is simply connected. It follows  that our varieties $X=F_r(Y)$ are simply connected. The Beauville-Bogomolov decomposition theorem~\cite[Théorème 1]{Beauville} gives a decomposition of $X \cong T \times W \times CY$, where $T$ is a complex torus, $W$ a product of hyper-Kähler manifolds, and $CY$ a product of strict Calabi-Yau manifolds. As shown in \cite[Théorème 3.4]{DebarreManivel}, the restriction morphism $H^i(\mathrm{Gr}(r+1, n+1), \mathbb{Q}) \to H^i(F_r(Y), \mathbb{Q})$ is injective for $i < \min\{\dim F_r(Y), n-2r-1\}$. Under the condition $n+1 = \binom{r+3}{2}$, $2 < \min\{\dim F_r(Y), n-2r-1\}$ for any $r \geq 2$. Thus, $H^{2,0}(X) = H^{2,0}(\mathrm{Gr}(r+1, n+1)) = 0$, and the Picard number of $X$ is $1$ for $r \geq 2$. Hence, in the Beauville-Bogomolov decomposition of $X$, only the strict Calabi-Yau manifold component remains, and it is irreducible since the Picard number is $1$.
\end{proof}

The distinct feature of Voisin's examples $X = F_r(Y)$ among all $K$-trivial manifolds revolves around the presence of a self-rational map, $\Psi: X \dashrightarrow X$, called the Voisin map. This map was constructed in~\cite{KCorr} as follows: Consider a general point $x \in X$, representing an $r$-dimensional linear space $P_x$ within $Y$. As shown in~\cite[Lemma 8]{KCorr}, there exists a unique $(r+1)$-dimensional linear subspace $H_x$ in $\mathbb{P}^n$ tangent to $Y$ along $P_x$. The intersection $H_x \cap Y$ forms a cubic hypersurface containing $P_x$ doubly, leaving a residual $r$-dimensional linear subspace in $Y$ represented by a point $x' \in X$. Then we set $\Psi(x) = x'$.

\subsection{Conjectures on Chow groups}
This section provides the necessary background and motivation for Conjecture~\ref{ConjVoisinCY}, with our analysis closely following the insights and frameworks presented in \cite{BBFiltration}, \cite{VoisinCitrouille}, and \cite{VoisinCoisotrope}.

Bloch and Beilinson have conjectured the existence of a descending filtration, denoted as $F^i CH_k(X)_\mathbb{Q}$ and called Bloch-Beilinson filtration, on the Chow groups with rational coefficients for any smooth complex projective variety $X$ subject to a set of axioms. In the case of zero-cycles their conjecture takes the following form.

\begin{conjecture}[Generalized Bloch Conjecture for $0$-Cycles]\label{ConjGBC0Cycle}
Given a correspondence $Z\in CH^n(X \times Y)_\mathbb{Q}$ between smooth projective varieties $X$ and $Y$, both of dimension $n$, if the map $[Z]^*: H^{i, 0}(Y) \to H^{i, 0}(X)$ vanishes for some $i \leq n$, then the pushforward $Z_*: Gr_F^iCH_0(X)_\mathbb{Q} \to Gr_F^iCH_{m-k}(Y)_\mathbb{Q}$ also vanishes for that $i$. Here, $F^\bullet$ represents the Bloch-Beilinson filtration and $Gr_F^\bullet$ its graded part.
\end{conjecture}

The original Bloch conjecture stated in~\cite{Bloch} is Conjecture~\ref{ConjGBC0Cycle} in the special case where $n = 2$. That is the reason why Conjecture~\ref{ConjGBC0Cycle} is named the generalized Bloch conjecture.

Now let $X$ be a projective hyper-Kähler manifold. Beauville's splitting conjecture, as introduced in~\cite{BeauvilleSplitting}, suggests that the Bloch-Beilinson filtration on the Chow ring of $X$ admits a natural multiplicative splitting. One weak version is now often referred to as the "weak splitting conjecture," detailed in the introduction of~\cite{VoisinCoisotrope}.

\begin{conjecture}[Beauville's Weak Splitting Conjecture~\cite{BeauvilleSplitting}]\label{ConjBeauvilleWeakSplitting}
For a projective hyper-Kähler manifold $X$, the cycle class map is injective on the subalgebra of $CH^*(X)$ that is generated by divisors.
\end{conjecture}

This conjecture was further expanded by Voisin in~\cite{0CycleHK} to include not only divisors but also Chern classes within the generating elements of the subalgebra.

\begin{conjecture}[Voisin~\cite{0CycleHK}]\label{conjVoisinChern}
For a projective hyper-Kähler manifold $X$, consider $C^*$ to be the subalgebra of $CH^*(X)$ generated by divisors and Chern classes. The cycle class map is injective on $C^*$.
\end{conjecture}

However, generalizing these conjectures to strict Calabi-Yau manifolds is known to be false. In~\cite[Examples 1.7]{BeauvilleSplitting}, a strict Calabi-Yau threefold $X$ whose cycle class map $cl: CH_1(X)_\mathbb{Q} \to H^4(X,\mathbb{Q})$ fails to be injective on the subgroup generated by intersections of divisors is constructed,  invaliding the ``strict Calabi-Yau version" of Conjectures~\ref{ConjBeauvilleWeakSplitting} and~\ref{conjVoisinChern}. Despite the counterexamples, it remains anticipated that Conjecture~\ref{conjVoisinChern} holds true for $0$-cycles in the context of strict Calabi-Yau manifolds. This expectation takes the form of Conjecture~\ref{ConjVoisinCY}, with compelling evidence provided by the constructions in~\cite[Theorem 1.2]{Bazhov}.

\begin{theorem}[Bazhov~\cite{Bazhov}]\label{ThmBazhov}
Let $Y$ be a projective homogeneous variety of dimension $n+1\geq 4$, and let $X$ be a general element of the anti-canonical system $|-K_Y|$. Then $X$ is a strict Calabi-Yau manifold satisfying Conjecture~\ref{ConjVoisinCY}.
\end{theorem}

\subsection{Main results of the work}\label{SectionVoisinMaps}


Our first result is a computation of two basic invariants of the Voisin map $\Psi$.

\begin{theoremA}\label{theoremA}
Let $X=F_r(Y),\, r\geq 0$ be as in \ref{SectionVoisinExample}, and let $\Psi: X\dashrightarrow X$ be the Voisin map. Then
\begin{itemize}
\item[(i)] For any $\omega \in H^0(X,K_X)$, we have 
\[\Psi^*\omega = (-2)^{r+1}\omega.\]
\item[(ii)] The map $\Psi: X \dashrightarrow X$ has a degree of $4^{r+1}$.
\end{itemize}
\end{theoremA}

\begin{remark}
For $r=1$, the results have been previously established in~\cite{KCorr}.
\end{remark}


Having Theorem A,  Conjecture~\ref{ConjGBC0Cycle} leads us to the following

\begin{conjecture}\label{PropConditionalActionOfPsiOnChow}
    Let $X = F_r(Y)$ with $r \geq 2$. Then for any $z \in CH_0(X)_{hom}$,
    \[
    \Psi_*z = (-2)^{r+1}z.
    \]
\end{conjecture}

Indeed, let us explain how Conjecture~\ref{ConjGBC0Cycle} implies Conjecture~\ref{PropConditionalActionOfPsiOnChow}. Let $N = \dim X$. As $X$ is a strict Calabi-Yau manifold, we have $H^{k, 0}(X) = 0$ for all $0 < k < N$. Consider $\Delta_X \in CH^n(X \times X)$ as the diagonal of $X \times X$. Then, for $0 < k < N$, $[\Delta_X]^*|_{H^{k,0}(X)} = 0$. Conjecture~\ref{ConjGBC0Cycle} implies that $\Delta_{X*}|_{Gr^k_FCH_0(X)} = 0$ for $0 < k < N$, leading to $Gr^k_FCH_0(X) = 0$ for these values of $k$. Thus, the Bloch-Beilinson filtration on $CH_0(X)$ simplifies to:
    \[
    0 = F^{N+1}CH_0(X) \subseteq F^NCH_0(X) = \ldots = F^1CH_0(X) = CH_0(X)_{hom} \subseteq F^0CH_0(X) = CH_0(X).
    \]
    Now, define $Z = \Gamma_\Psi - (-2)^{r+1}\Delta_X$. Given Theorem A, $[Z]^*|_{H^{N, 0}(X)} = 0$. Applying Conjecture~\ref{ConjGBC0Cycle} once more, we find $Z_*|_{Gr^N_FCH_0(X)} = 0$, confirming that $Z_*$ is zero on  $Gr^N_FCH_0(X) = F^NCH_0(X) = F^1CH_0(X) = CH_0(X)_{hom}$, yielding the desired conclusion.

Our second main result is the proof of Conjecture~\ref{PropConditionalActionOfPsiOnChow} when $r=2$
\begin{theoremB}\label{theoremB}
    Let $Y\subset \mathbb P^9$ be  a general cubic $8$-fold, let $X=F_2(X)$, and let $\Psi: X\dashrightarrow X$ be the Voisin map. Then, for any $z\in CH_0(X)_{hom}$, we have
$$\Psi_*(z)=-8z\,\,  {\rm  in\, }  CH_0(X)_{hom}.$$
\end{theoremB}

In the proof of Theorem B, we will use the notion of constant cycle subvariety~\cite{ConstantCycle, VoisinCoisotrope}. Let us recall the definition.

\begin{definition}[\cite{ConstantCycle, VoisinCoisotrope}]
    Let $X$ be a smooth algebraic variety. A closed subvariety $j: Z \hookrightarrow X$ is a \emph{constant-cycle subvariety} if every two points $z_1, z_2 \in Z$ are rationally equivalent in $X$. Equivalently, the image of the morphism, $j_*: CH_0(Z) \to CH_0(X)$, is $\mathbb{Z}$.
\end{definition}

Let $F$ denote the closure of the fixed locus under the Voisin map $\Psi: X \dashrightarrow X$. We next observe that   Conjecture~\ref{PropConditionalActionOfPsiOnChow} immediately  leads us to   the following

\begin{conjecture}\label{conjFixedLocusConstantCycle}
     The variety $F\subset X$ is a constant-cycle subvariety for $r\geq 2$.
\end{conjecture}

Let us explain how  Conjecture~\ref{conjFixedLocusConstantCycle} is implied by  Conjecture~\ref{PropConditionalActionOfPsiOnChow}. Consider two points $x, y \in F$. Assuming $x, y$ are within the fixed locus of $\Psi$, we have $\Psi(x) = x$ and $\Psi(y) = y$, leading to $\Psi_*(x - y) = x - y$. However, Conjecture~\ref{PropConditionalActionOfPsiOnChow} indicates $\Psi_*(x - y) = (-2)^r(x - y)$. Consequently, $x - y = 0 \in CH_0(X)_\mathbb{Q}$. Roitman's theorem~\cite{Roitman} implies $CH_0(X)_{hom}$ is torsion-free due to the triviality of $Alb(X)$, so it follows that $x = y \in CH_0(X)$. Thus, $F \subset X$ forms a constant-cycle subvariety.

In the present paper, we will prove  directly  Conjecture~\ref{conjFixedLocusConstantCycle} for $r=2$, and this will be   one step in our proof of Theorem B, that is, Conjecture~\ref{PropConditionalActionOfPsiOnChow} for $r=2$.

\begin{remark}
    Interestingly, even in the scenario of $r=1$—not addressed in Conjecture~\ref{conjFixedLocusConstantCycle}—it has been proved in~\cite{0CycleHK} that $F \subset X$ is a constant-cycle subvariety.
\end{remark}


Let Ind be the indeterminacy locus of the Voisin map $\Psi: X\dashrightarrow X$. We prove as a consequence of Theorem B the following result

\begin{theoremC}
    Assume $r=2$. Then any $0$-cycle of $X$ which is a polynomial in the Chern classes of $X$ and divisor classes is rationally equivalent to a cycle supported on Ind.
\end{theoremC}
\begin{corollary}
    If $\mathrm{Ind}\subset X$ is a constant cycle subvariety, Conjecture~\ref{ConjVoisinCY} holds true for $X = F_2(Y)$.
\end{corollary}

This leaves us with an open
\begin{question}\label{QuestionConstantCycleInd}
    Is $\mathrm{Ind} \subset X$ a constant cycle subvariety for $r \geq 2$?
\end{question}

\begin{remark}
    In the hyper-Kähler case ($r=1$), the indeterminacy locus $\mathrm{Ind}_0$ is \emph{not} a constant-cycle subvariety, nor is it Lagrangian, as indicated by~\cite[Lemma 2]{Amerik} and~\cite{KCorr}.
\end{remark}

Theorem A, B and C will be proved in Section~\ref{SectionActionOnForms}, \ref{SectionActionOnChow} and \ref{SectionProofOfC}, respectively.

\section*{Acknowledgments}

I extend my deepest gratitude to Claire Voisin, my PhD advisor, for proposing this both enriching and fascinating subject and for guiding through for the project. Her revision of previous versions of the article has greatly improved the presentation. It has been an immense honor to journey through my academic pursuits under her mentorship. I would like to thank Katia Amerik for her meticulous revision. I would like to express my gratitude to Frank Gounela for his constructive suggestions and for bringing the article~\cite{FixedLocusII} to our attention. I am also grateful to Jieao Song for providing the code snippets in Appendix~\ref{SectionCalculations} that have not only simplified the calculations but also significantly enhanced the presentation of this article compared to its earlier versions. Special appreciation is directed to Pietro Beri, Laurent Manivel, Mauro Varesco, and Zhi Jiang for their invaluable support and encouragement. This work was conducted in the stimulating environment of the Institut de Mathématiques de Jussieu - Paris Rive Gauche, to which I owe my thanks for providing an excellent research atmosphere. Additionally, my PhD thesis, within which this work was developed, received support from the ERC Synergy Grant HyperK (Grant Agreement No. 854361), an acknowledgment I proudly share.

Thank you all for your indispensable contributions to my academic journey.

\section{Action of the Voisin map on top degree holomorphic forms}\label{SectionActionOnForms}

In this section, we compute the action $\Psi^*: H^0(X,K_X)\rightarrow H^0(X,K_X)$ of the Voisin map and its degree, proving Theorem A.

\begin{remark}
    For $r=0$, where $Y$ is a plane cubic curve and $X=F_0(Y)=Y$ denotes an elliptic curve, $\Psi: X \to X$ acts by mapping $x$ to $-2x$, in accordance with the addition law of the elliptic curve. Consequently, the degree of $\Psi: X \to X$ is $4$. In the scenario where $r=1$ and $X$ is a hyper-Kähler fourfold, the result $\deg \Psi = 16$ is first discovered by Voisin~\cite{KCorr}, and can be derived using either Chow-theoretic techniques~\cite{KCorr}, \cite[Corollary 1.7]{AmerikVoisin} or vector bundle methods \cite[Lemma 4.12 and Proposition 4.17]{Cubic}.
\end{remark}

Let us first note the following
\begin{lemma}\label{LmmNearlyEquiv}
    The two assertions in Theorem A are equivalent, up to sign.
\end{lemma}
 \begin{proof}
     (i) leads to (ii) as follows: Given that $\sigma \wedge \bar\sigma$ constitutes a volume form on $X$, it follows that
    \[
    \deg \Psi \int_X \sigma \wedge \bar\sigma = \int_X \Psi^* \sigma \wedge \Psi^* \bar\sigma = \int_X (-2)^{r+1} \sigma \wedge (-2)^{r+1} \bar\sigma = 4^{r+1} \int_X \sigma \wedge \bar\sigma.
    \]
    Conversely, let us show that (ii) implies $\Psi^* \sigma = \pm 2^{r+1} \sigma$, aligning closely with (i). The rational map can be consistently defined across the family of cubic hypersurfaces, implying the degree of $\Psi$ remains invariant across different selections of the generic cubic hypersurface $Y$. Therefore, we may presume $Y$ is defined over $\mathbb{Q}$. Consequently, the rational map $\Psi: X \dashrightarrow X$ is also defined over $\mathbb{Q}$. Thus, $\Psi^*: H^0(X, K_{X/\mathbb{Q}}) \to H^0(X, K_{X/\mathbb{Q}})$ operates by multiplication by a rational number, given that $H^0(X, K_{X/\mathbb{Q}})$ is a one-dimensional $\mathbb{Q}$-vector space. If $\Psi^*\sigma = \lambda \sigma$ with $\lambda \in \mathbb{Q}$, then $\lambda^2 = 4^{r+1}$, leading to $\lambda = \pm 2^{r+1}$.
 \end{proof}

\subsection{Proof of Theorem A (i)}
Let $\mathrm{Fix}(\Psi):=\{x\in X: \Psi \textrm{ is defined at } x \textrm{ and } \Psi(x) = x\}$ be the fixed locus of $\Psi: X\dashrightarrow X$.
\begin{proposition}
    For $Y$ general, the fixed locus $\mathrm{Fix}(\Psi)$ is not empty and is of codimension $r+1$ in $X$.
\end{proposition}
\begin{proof}
    We will suppose that $r\geq 2$ since the case $r=1$ is shown in~\cite[Proposition 3.1]{AmerikBogomolovRovinsky}  already (see also~\cite[Corollary 3.13]{FixedLocus}). Let $B=\mathbb PH^0(\mathbb P^n, \mathcal O(3))$ be the parametrizing space of cubic hypersurfaces. Let $\mathrm{Fl} = \{(t, s)\in \mathrm{Gr}(r+1, n+1)\times \mathrm{Gr}(r+2, n+1): P_t\subset \Pi_s\}$ be the flag variety of pairs of linear subspaces in $\mathbb P^n$. Let 
    \[\mathcal I = \{(f, t, s)\in B\times \mathrm{Fl}: \Pi_s\textrm{ is the only }\mathbb P^{r+1}\textrm{ such that } Y_f\cap \Pi_s = 3P_t\},
    \]
    \[
    \tilde{\mathcal I} = \{
    (f, t, s)\in B\times \mathrm{Fl}: Y_f\cap \Pi_s \supset 3P_t
    \}.
    \]
     Let $p: \mathcal I\to B$ (resp. $\tilde p: \tilde{\mathcal I}\to B$) and $q: \mathcal I\to \mathrm{Fl}$ (resp. $\tilde q: \tilde{\mathcal I}\to \mathrm{Fl}$) be the canonical projection maps. By definition of $\Psi$, for any $f\in B$, the fiber $p^{-1}(f)\subset \mathcal I$ coincides with the fixed locus of $\Psi$ for the cubic hypersurface $Y_f$. Hence, it suffices to show that the map $p:
    \mathcal I\to B$ is dominant and that the dimension of the general fibers is $\dim X - r - 1$.
    \begin{lemma}\label{LmmVarietyMathcalI}
        The variety $\mathcal I$ is open dense in $\tilde{\mathcal I}$. The dimension of $\mathcal I$ is $\dim X + \dim H^0(\mathbb P^n, \mathcal O(3)) - r - 2$.
    \end{lemma}
    \begin{proof}
        Let us consider the fibers of the map $\tilde q: \tilde{\mathcal I}\to \mathrm{Fl}$. For any element $(t,s)\in \mathrm{Fl}$ representing linear subspaces $P$ and $\Pi$ respectively, we may assume without loss of generality that 
        \[
    P=\{(x_0, x_1,\ldots, x_r, 0,\ldots, 0)\},
    \]
    \[
    \Pi=\{(x_0, x_1, \ldots, x_r, x_{r+1}, 0,\ldots,0)\}.
    \]
    The fiber $\tilde q^{-1}((t, s))$ parametrizes the cubic hypersurfaces $Y$ such that $\Pi\cap Y\supset 3P$. The last condition implies that the defining equation $f$ of $Y$ in $\mathbb P^n$ is given by 
    \[
    f(Y_0,\ldots, Y_n) = \alpha Y_{r+1}^3 + Y_{r+2}Q_{r+2} + \ldots + Y_nQ_n,
    \]
    where $\alpha\in \mathbb C$ is a constant and $Q_{r+1},\ldots, Q_n$ are quadratic polynomials. To write the above fact more formally, the fiber $\tilde q^{-1}((t, s))$ is identified with the image of the following map
    \[
    \begin{array}{cccc}
        \Phi: & (\mathbb C\oplus H^0(\mathbb P^n, \mathcal O(2))^{\oplus n-r-1}) - \{0\} & \to & B  \\
         & (\alpha, Q_{r+1}, \ldots, Q_n) & \mapsto & \alpha Y_{r+1}^3 + Y_{r+2}Q_{r+2} + \ldots + Y_nQ_n.
    \end{array}
    \]
    It is not hard to see that $\mathcal I\cap \tilde q^{-1}((t, s))$ parametrizes the cubic hypersurfaces $Y$ whose defining equation $f$ is given by 
    \[ f(Y_0,\ldots, Y_n) = \alpha Y_{r+1}^3 + Y_{r+2}Q_{r+2} + \ldots + Y_nQ_n,
    \]
    subject to $\alpha \neq 0$ and $Q_{r+1}(Y_0, \ldots, Y_r, 0, \ldots, 0), \ldots, Q_n(Y_0, \ldots, Y_r, 0, \ldots, 0)$ are linearly independent in $H^0(\mathbb P^r, \mathcal O(2))$. Notice that $n - r - 1 = \dim H^0(\mathbb P^r, \mathcal O(2))$, the above conditions give an open dense subset in $(\mathbb C\oplus H^0(\mathbb P^n, \mathcal O(2))^{\oplus n-r-1}) - \{0\}$, which implies that $\mathcal I\cap \tilde q^{-1}((t, s))$ is open dense in $\tilde q^{-1}((t, s))$ for any $(t, s)\in \mathrm{Fl}$. Therefore, $\mathcal I$ is open dense in $\tilde{\mathcal I}$.

    For the next step, let us calculate the dimension of $\mathcal I$. To this end, we use now the projection $\mathcal I \rightarrow \mathrm{Fl}$ and compute the dimensions of its fibers. Since $\mathrm{Fl}$ is a $\mathbb P^{r+1}$-bundle over $\mathrm{Gr}(r+2, n+1)$, it is clear that $\dim \mathrm{Fl} = (r+2)(n-r-1) + (r+1)$. By the above description of the fiber $q^{-1}((t, s))$, the dimension of the fiber of $q: \mathcal I\to \mathrm{Fl}$ is the dimension of the space of cubic polynomials in $Y_0, \ldots, Y_n$ such that each monomial of it contains either one of the variables $Y_{r+2}, \ldots, Y_n$, and the latter is $\dim H^0(\mathbb P^n, \mathcal O(3)) - \dim H^0(\mathbb P^{r+1}, \mathcal O(3))$. Taken together, we have
    \[
    \dim {\mathcal I} = (r+2)(n-r-1) + (r+1) + \dim H^0(\mathbb P^n, \mathcal O(3)) - \dim H^0(\mathbb P^{r+1}, \mathcal O(3)).
    \]
    Taking into account of the relation $n+1 = \dim H^0(\mathbb P^{r+1}, \mathcal O(2))$ and the fact that $\dim X = (n-r)(r+1) - \dim H^0(\mathbb P^r, \mathcal O(3))$, we find that $\dim {\mathcal I} = \dim X + \dim H^0(\mathbb P^n, \mathcal O(3)) - r - 2$, as desired.
    \end{proof}
    \begin{lemma}
        The map $\tilde p: \tilde{\mathcal I}\to B$ is surjective.
    \end{lemma}
    \begin{proof}
        It is equivalent to showing that for any cubic hypersurface $Y$, there are linear subspaces $P, \Pi\subset\mathbb P^n$ of dimension $r$, $r+1$ respectively such that $\Pi\cap Y\supset 3P$. To construct such examples, it is worthwhile to notice that for $r\geq 2$, the variety $F_{r+1}(Y)$ is non-empty by a dimension counting argument. For any $\Pi\in F_{r+1}(Y)$ and any $P\subset \Pi$, we have $\Pi\cap Y\supset 3P$. That concludes the proof.
    \end{proof}
    By the two lemmas above, we conclude that the map $p: \mathcal I\to B$ is dominant. Since the map $p: \mathcal I\to B$ is dominant, the dimension of the general fiber is $\dim\mathcal I - \dim B $, which is equal to $\dim X - r - 1$.
\end{proof}

\begin{proposition}\label{PropEigenPoly}
Let $x\in X$ be a generic fixed point of $\Psi: X\dashrightarrow X$ representing an $r$-dimensional linear subspace $P\subset Y$. Then the linear map
\[
\Psi_{*,x}: T_{X,x}\to T_{X, x}
\]
 has $N-r-1$ eigenvalues equal to $1$, corresponding to the tangent space to the fixed locus $F$ of $\Psi$ at $x$, and $r+1$ eigenvalues equal to $-2$, corresponding to the action of $\Psi_*$ in the normal direction to $F$ at $x$.
\end{proposition}
\begin{proof}
    Since $x\in X$ is a generic fixed point of $\Psi$ representing a plane $P\subset Y$, there is a unique $(r+1)$-plane $\Pi\subset \mathbb P^n$ such that $\Pi\cap Y=3P$ as algebraic cycles. Without loss of generality, we may assume that 
    \[
    P=\{(x_0, x_1,\ldots, x_r, 0,\ldots, 0)\},
    \]
    \[
    \Pi=\{(x_0, x_1, \ldots, x_r, x_{r+1}, 0,\ldots,0)\}.
    \]
    The fact $\Pi\cap Y=3P$ implies that the defininig equation $f$ of $Y$ in $\mathbb P^n$ is given by
    \begin{equation}\label{EqDefiningEquationOfY}
        f(Y_0, \ldots, Y_n)=Y_{r+1}^3+Y_{r+2}Q_{r+2}+\ldots+ Y_nQ_n,
    \end{equation}
    where $Q_{r+1}, \ldots, Q_n$ are quadratic polynomials. Let $v_1\in T_{X, x}$ be a nonzero tangent vector given by a path $\{P_{1,t}\}_{t\in \Delta}$ with $P_{1,0}=P$, where $\Delta$ is the unit disc in the complex plane. To understand $\Psi_{*, x}(v_1)\in T_{X, x}$, we have to determine $P_{2,t}:=\Psi(P_{1,t})$ for each small $t\in \Delta$. By the construction, for each small $t\in \Delta$, there is a unique $(r+1)$-plane $\Pi_t\subset \mathbb P^n$ such that $\Pi_t\cap Y=2P_{1,t}+P_{2,t}$. Since the inclusions $P_{1,t}\subset \Pi_t\subset \mathbb P^n$ vary holomorphically with respect to $t\in \Delta$, we may assume that $\Pi_t$ is given by the image of a linear map $G_t: \mathbb P^{r+1}\to \mathbb P^n$ varying holomorphically with $t\in \Delta$ and that 
    \begin{equation}
        P_{1,t}=\{G_t(x_0,\ldots, x_r, 0)\}.
    \end{equation}
    Then the relation $\Pi_t\cap Y= 2 P_{1,t}+P_{2,t}$ can be translated as 
    \begin{equation}\label{EqFormOfFcircGt}
        f\circ G_t(x_0, \ldots, x_{r+1})=x_{r+1}^2L_t(x_0, \ldots, x_{r+1})
    \end{equation}
    for some linear function $L_t$ and that 
    \begin{equation}\label{EqDefiningEquationOfP2t}
    P_{2,t}=\{G_t(x_0,\ldots, x_{r+1}): L_t(x_0, \ldots, x_{r+1})=0\}.
    \end{equation}
    Let us simplify the notation and write $\mathbf x=(x_0,\ldots, x_{r+1})$. Let us write $G_t(\mathbf x)$ as 
    \begin{equation}\label{EqExpansionOfGt}
        G_t(\mathbf x) = (x_0+tY_0(\mathbf x), x_1+tY_1(\mathbf x), \ldots, x_{r+1}+tY_{r+1}(\mathbf x),\\ tY_{r+2}(\mathbf x), \ldots, tY_n(\mathbf x))+O(t^2).
    \end{equation}
    Then by (\ref{EqDefiningEquationOfY}), we have
    \begin{equation}\label{EqExpansionOfFcircGt}
        f\circ G_t(\mathbf x)=x_{r+1}^3+ 3tx_{r+1}^2Y_{r+1}(\mathbf x)+t\sum_{i=r+2}^{n}Y_i(\mathbf x)Q_i(\mathbf x, \Vec{0})+O(t^2).
    \end{equation}
    Comparing (\ref{EqFormOfFcircGt}) and (\ref{EqExpansionOfFcircGt}), we get that $x_{r+1}^2$ divides $\sum_{i=r+2}^n Y_i(\mathbf x)Q_i(\mathbf x, \Vec{0})$, and that 
    \begin{equation}
        L_t(\mathbf x)=x_{r+1}+t\left\{3Y_{r+1}(\mathbf x)+\frac{Y_i(\mathbf x)Q_i(\mathbf x, \Vec{0})}{x_{r+1}^2}\right\}+O(t^2).
    \end{equation}
    Hence, the equation $L_t(\mathbf x)=0$ has an explicit expression as 
    \begin{equation}\label{EqExplicitFormOfLt}
        x_{r+1}=-t\left\{
        3Y_{r+1}(\mathbf x', 0)+\frac{\sum_{i=r+2}^n Y_iQ_i}{X_{r+1}^2}(\mathbf x', 0)
        \right\}+O(t^2),
    \end{equation}
    where $\mathbf x'=(x_0,\ldots, x_r)$. Comparing (\ref{EqDefiningEquationOfP2t})(\ref{EqExpansionOfGt})(\ref{EqExplicitFormOfLt}), we get
    \begin{align*}
        P_{2,t}=&\{(x_0+tY_0(\mathbf x', 0), \ldots, x_r+tY_r(\mathbf x', 0), -2tY_{r+1}(\mathbf x',0)-t\frac{\sum Y_iQ_i}{X_{r+1}^2}(\mathbf x, 0), \\
        &tY_{r+2}(\mathbf x',0), \ldots, tY_n(\mathbf x',0))+O(t^2)\}.
    \end{align*}
    We view now $v_1\in T_{X, x}$ as a $(r+1)\times (n-r)$-matrix via the following natural inclusion and identification $T_{X,x}\subset T_{\mathrm{Gr}(r+1, n+1), x}=Hom(\mathbb C^{r+1}, \mathbb C^{n+1}/\mathbb C^{r+1})$. Then $v_1$ is represented by the matrix 
    \[
    \begin{pmatrix}
     Y_{r+1}\\
     Y_{r+2}\\
     \vdots\\
     Y_{n}
    \end{pmatrix},
    \]
    where we view each $Y_i$ as a row vector $(a_0, a_1,\ldots, a_r)$ if $Y_i(x_0, \ldots, x_r, 0)=a_0x_0+\ldots +a_rx_r$. By the above calculations, we see that $\Psi_{*, x}(v_1)$ is represented by the matrix 
    \[
    \begin{pmatrix}
        -2Y_{r+1}-\frac{\sum_{i=r+2}^n Y_iQ_i}{X_{r+2}^2}\\
        Y_{r+2}\\
        \vdots\\
        Y_n
    \end{pmatrix}
    \]
    with the same explanation about the notations. Hence, $\Psi_{*,P}$ sends 
    $
    \begin{pmatrix}
     Y_{r+1}\\
     Y_{r+2}\\
     \vdots\\
     Y_{n}
    \end{pmatrix}
    $
    to 
    $
    \begin{pmatrix}
        -2Y_{r+1}-\frac{\sum_{i=r+2}^n Y_iQ_i}{X_{r+2}^2}\\
        Y_{r+2}\\
        \vdots\\
        Y_n
    \end{pmatrix}.
    $
    The representing matrix of $\Psi_{*, x}$ is uppper triangular with the diagonal $r+1$ copies of $-2$ and $N - r - 1$ copies of $1$. Therefore, the eigenpolynomial of $\Psi_{*,x}:T_{X, x}\to T_{X,x}$ is $(t+2)^{r+1}(t-1)^{N-r-1}$, as desired.
\end{proof}

\begin{proof}[Proof of Theorem A (i)]
    By Proposition~\ref{PropEigenPoly}, the eigenpolynomial of $\Psi^*_P: \Omega_{X,P}\to \Omega_{X, x}$ is $(t+2)^{r+1}(t-1)^{N-r-1}$ for a generic fixed point $P$ of $\Psi$. Hence, the map $\Psi^*_P: K_{X,P}\to K_{X,P}$ is given as the multiplication by $(-2)^{r+1}$. Let $\omega\in H^0(X,K_X)$ be a nowhere zero top degree holomorphic differential form on $X$. Then, as $H^0(X,K_X)$ has dimension $1$,
 \begin{equation}\label{EqARelation}
        \Psi^*\omega=\lambda\omega
    \end{equation} 
    for some $\lambda\in \mathbb C$. By evaluating the equation (\ref{EqARelation}) on the point $P\in X$, we find $\lambda=(-2)^{r+1}$, and we finish the proof of Theorem A (i).
\end{proof}

\subsection{A direct proof of Theorem A (ii)}
While Theorem A (ii) has been established as a consequence of Theorem A (i) using Lemma~\ref{LmmNearlyEquiv}, we present a direct proof, using an enumerative geometry viewpoint.
\begin{proof}[Proof of Theorem A
(ii)]
The beginning of the proof is similar to that of \cite[Lemma 4.12]{Cubic}. Let $P'\in X$ be a generic $r$-plane in $Y$. The preimage of $P'$ via $\Psi$ is the set of $r$-planes $P$ in $Y$ such that there is an $(r+1)$-plane $\Pi$ in $\mathbb P^n$ such that $\Pi\cap Y=2P+P'$ as algebraic cycles. In general, let $\Pi\subset \mathbb P^n$ be an $(r+1)$-plane containing $P'$. Then $\Pi\cap Y=Q\cup P'$ where $Q\subset \Pi$ is a quadratic hypersurface that corresponds to a quadratic form $q_{\Pi}$. The preimage of $P'$ via $\Psi$ corresponds to the $(r+1)$-planes $\Pi$ such that the quadratic form $q_{\Pi}$ is of rank $1$. Now let $\pi: \mathbb P^n\dashrightarrow \mathbb P^{n-r-1}$ be the projection map induced by the $r$-plane $P'$. One can resolve the indeterminacies of $\pi:\mathbb P^n\dashrightarrow \mathbb P^{n-r-1}$ by blowing up $\mathbb P^n$ along $P'$. Let $p:\mathcal P\to \mathbb P^n$ be the blowup map. Then the induced map $q: \mathcal P\to \mathbb P^{n-r-1}$ is a projective bundle induced by a vector bundle $\mathcal E$ over $\mathbb P^{n-r-1}$. It is proven in~\cite[Section 1.5.1-1.5.2]{Cubic} the following
\begin{lemma}
    (i) We have $\mathcal E\cong \mathcal O_{\mathbb P^{n-r-1}}^{r+1}\oplus \mathcal O_{\mathbb P^{n-r-1}}(-1)$.\\
    (ii) There is a section $q\in H^0(\mathbb P^{n-r-1}, Sym^2\mathcal E^*\otimes \mathcal O_{\mathbb P^{n-r-1}}(1))$ such that on each point $\Pi\in\mathbb P^{n-r-1}$, $q$ coincides with $q_{\Pi}$.
\end{lemma}
Hence, to calculate the degree of $\Psi: X\dashrightarrow X$, we need to calculate the number of points in $\mathbb P^{n-r-1}$ over which $q$ is of rank $1$. Since $P'$ is generic, by a dimension counting argument we find that there is no plane $\Pi$ containing $P'$ that is contained in $Y$. Therefore, we are reduced to calculate the degree of the locus where a generic section of $Sym^2\mathcal E^*\otimes \mathcal O(1)$ has rank $\leq 1$. Let $f: \mathbb P^{n-r-1}\to \mathbb P^{n-r-1}$ be a morphism of degree $2^{n-r-1}$ (e.g., $f: (x_0,\ldots, x_{n-r-1})\mapsto (x_0^2, \ldots, x_{n-r-1}^2)$). Let $Z_q\subset \mathbb P^{n-r-1}$ be the locus where $q\in H^0(\mathbb P^{n-r-1}, Sym^2\mathcal E^*\otimes \mathcal O(1))$ is of rank $\leq 1$. Then $f^*Z_q$ is the locus where $f^*q\in H^0(Sym^2\mathcal F)$ is of rank $\leq 1$. Here, $\mathcal F=\mathcal O(1)^{r+1}\oplus \mathcal O(3)$. Then 
\begin{equation}\label{EqRelationOfF*}
    \deg f^*Z_q=\deg f\cdot\deg Z_q=2^{n-r-1}\deg Z_q.
\end{equation}
Let $p: \mathbb P(\mathcal F)\to \mathbb P^{n-r-1}$ and $\pi: \mathbb P(\mathrm{Sym}^2\mathcal F)\to \mathbb P^{n-r-1}$ be the projective bundles over $\mathbb P^{n-r-1}$ corresponding to the vector bundles $\mathcal F$ and $\mathrm{Sym}^2\mathcal F$, respectively. Consider the Veronese map of degree $2$ over the projective space $\mathbb P^{n-r-1}$.
\[\begin{array}{cccc}
v_2: &\mathbb P(\mathcal F) &\to &\mathbb P(\mathrm{Sym}^2\mathcal F)\\
& \alpha &\mapsto &\alpha^2.
\end{array}\]
The section $f^*q$ induces a section $\sigma_q$ of the projective bundle $\pi: \mathbb P(Sym^2\mathcal F)\to\mathbb P^{n-r-1}$ such that $f^*Z_q$ coincides with $\pi_*(\mathrm{Im}\sigma_q\cap \mathrm{Im}(v_2)) = p_*(v_2^*(\mathrm{Im}(\sigma_q)) $. Let $\mathcal S$ (resp. $\mathcal S'$) be the tautological subbundle of $\mathbb P(\mathcal F)$ (resp. $\mathbb P(Sym^2\mathcal F)$). Let $\mathcal Q'$ be the tautological quotient bundle of $\mathbb P(Sym^2\mathcal F)$. Then the pull-back of $f^*q\in H^0(\mathbb P^{n-r-1}, Sym^2\mathcal F)$ gives a section of $H^0(\mathbb P(Sym^2\mathcal F), \mathcal Q')$ whose zero locus coincides with $\mathrm{Im}(\sigma_q)$. Writing $c(\mathcal Q')$ as $\pi^*c(\mathrm{Sym}^2\mathcal F)\cdot c(\mathcal S')^{-1}$ and noticing that $v_2^*\mathcal S'=\mathcal S^{\otimes 2}$, we see that $f^*Z_q$ is the part of degree $n$ of $p_*(c(S^{\otimes 2})^{-1}\cdot p^* c(\mathrm{Sym}^2\mathcal F))$. Let $s(\mathcal F)$ be the formal series of the Segre classes of $\mathcal F$~\cite[Chapter 3]{Fulton} and let $h=c_1(\mathcal O_{\mathbb P^n}(1))$. By a direct calculation, we find that
\begin{align*}
    f^*Z_q & = 2^{r+1}\sum_i2^{i}s_i(\mathcal F)\cdot c_{n-r-1-i}(\mathrm{Sym}^2\mathcal F)\\
    & = \textrm{The degree $n-r-1$ part of } 2^{r+1} \frac{(1+2h)^{\frac{(r+1)(r+2)}{2}}\cdot (1+4h)^{r+1}\cdot (1+6h)}{(1+2h)^{r+1}\cdot (1+6h)}\\
    &= 2^{r+1}\cdot 2^{\frac{(r+1)r}{2}}\cdot 4^{r+1} h^{n-r-1}.
\end{align*}
Taking into account of the relation (\ref{EqRelationOfF*}), we find that $\deg Z_q=4^{r+1}$, as desired.
\end{proof}

\section{Action of the Voisin map on $CH_0(X)_{hom}$}\label{SectionActionOnChow}
The main purpose of this section is to prove Theorem B. Part of the argument will work for any $r$ and will be given in Section~\ref{SectionDecompOfTheAction}.

\subsection{Decomposition of the action of $\Psi_*$}\label{SectionDecompOfTheAction}
Let $X=F_r(Y)$ be as presented in Section~\ref{SectionVoisinExample}. Let $I_r = \{(x, x')\in X\times X: \dim (P_x\cap P_{x'})\geq r - 1 \}$.
\begin{theorem}\label{ThmDecOfActionOfPsiGeneralCase}
    There is an $(r+1)$-cocycle $\gamma\in CH^{ r+1}(X)$ lying in the image of the restriction map $CH^{ r+1}(\mathrm{Gr}(r+1,n+1))\to CH^{ r+1}(X)$ such that for any $z\in CH_0(X)$, we have in $CH_0(X)_{\mathbb Q}$
    \begin{equation}\label{EqDecOfTheActionOfPsi}
        \Psi_*z=(-2)^{r+1}z+\gamma\cdot I_{r,*}(z).
    \end{equation}
\end{theorem}
\begin{proof}
    The argument is quite similar to that in the proof of~\cite[Theorem 4.16]{VoisinCitrouille}. Let $B=\mathbb PH^0(\mathbb P^n, \mathcal O(3))$ be the projective space parametrizing all cubic hypersurfaces in $\mathbb P^n$. For $f\in B$, we denote $Y_f\subset \mathbb P^9$ the hypersurface defined by $f$. There is a universal Fano variety of $r$-spaces defined as follows.
    \[
    \mathcal X:=\{(f,x)\in B\times \mathrm{Gr}(r+1,n+1): P_x\subset Y_f\}.
    \]
    The fiber of the projection map $\pi:\mathcal X\to B$ over a point $f\in B$ is the Fano variety of $r$-spaces of $Y_f$. The {Voisin} map $\Psi: X\dashrightarrow X$ can also be defined universally. In fact, we define the universal graph
    $\Gamma_{\Psi_{\mathrm{univ}}}$ as the closure of the set of pairs $(x, x')\in \mathcal X\times_B\mathcal X$ such that there exists a $(r+1)$-dimensional linear subspace $H_{r+1}\subset \mathbb P^n$ such that $H_{r+1}\cap Y_{\pi(x)}=2P_x+P_{x'}$. Denote
    \[
    \begin{tikzcd}
       \Gamma_{\Psi_{\mathrm{univ}}}\arrow[r, hookrightarrow, "i"] &\mathcal X\times_B\mathcal X\arrow{r}{q}\arrow{d}{p}& \mathrm{Gr}(r+1,n+1)\times \mathrm{Gr}(r+1,n+1)\\
        &B
    \end{tikzcd}
    \]
    the inclusion map and the natural projection maps. Then over a point $f\in B$ corresponding to a general smooth cubic hypersurface $Y_f$, the fiber $(p\circ i)^{-1}(f)$ is the graph of the {Voisin} map on $F_r(Y_f)$. There is a stratification of $\mathrm{Gr}(r+1,n+1)\times \mathrm{Gr}(r+1,n+1)$ given as follows. 
    \[
    \begin{tikzcd}
        \mathrm{Gr}(r+1,n+1)\times \mathrm{Gr}(r+1,n+1)\arrow[d, hookleftarrow]  \\
        I^G_1 \arrow[d, hookleftarrow] := \{(x,x'): P_x\cap P_{x'}\neq \emptyset\}\\
        I^G_2 \arrow[d, hookleftarrow]  :=\{(x,x'): P_x\cap P_{x'} \textrm{ contains a line }\}\\
        \vdots\arrow[d, hookleftarrow]\\
        I_r^G\arrow[d, hookleftarrow]:=\{(x,x'): P_x\cap P_{x'} \textrm{ contains an } (r-1) \textrm{-space}\}\\
        I_{r+1}^G=\Delta_G  :=\{(x,x)\}
    \end{tikzcd}
    \]
    {In other words, the subvariety $I_k^G$ is defined as} \[I_k^G=\{(x,x')\in \mathrm{Gr}(r+1, n+1)\times \mathrm{Gr}(r+1, n+1): \dim (P_x\cap P_{x'})\geq k-1\}.\]
    
    The map $q: \mathcal X\times_B\mathcal X\to \mathrm{Gr}(r+1, n+1)\times \mathrm{Gr}(r+1, n+1)$ { has a projective bundle structure over each stratum}. Precisely, let $d=\dim B$. Let $I_k^{\mathcal X}$ be the preimage of $I_k^G$ under the map $q: \mathcal X\times_B\mathcal X\to \mathrm{Gr}(r+1,n+1)\times \mathrm{Gr}(r+1,n+1)$. Over the open subset $I_{k}^G-I_{k+1}^G$, the map $q_{|I_{k}^\mathcal X}$ is a $\mathbb P^{d-\delta_k}$-bundle where $\delta_k=2h^0(\mathbb P^r,\mathcal O(3))-h^0(\mathbb P^{k-1},\mathcal O(3))$.  Let $i_k^{\mathcal X}: I_k^\mathcal X\hookrightarrow \mathcal X\times_B \mathcal X$ be the inclusion maps. The Chow ring of $\mathcal X\times_B\mathcal X$ is thus equal to
    \[\bigoplus_ih^i\cdot \left(\bigoplus_k \left(q_{|I_k^{\mathcal X}}\right)^*CH^*(I_k^G)\right),
    \]
    where $h=p^*\mathcal O_B(1)$.
    Now let us consider the Chow class of $\Gamma_{\Psi_{\mathrm{univ}}}$ in $CH^{N}(\mathcal X\times_B\mathcal X)$, where $N=(r+1)(n-r)-\binom{r+3}{3}$ is the dimension of $X$. By the construction, we have $\Gamma_{\Psi_{\mathrm{univ}}}\subset I_r^\mathcal X$, thus \[
    \Gamma_{\Psi_{\mathrm{univ}}}\in \bigoplus_i h^i\cdot \left(i_{r+1,*}^\mathcal X(q_{|\Delta_{\mathcal X}})^*CH^{-i}(\Delta_G) + i_{r*}^\mathcal X\left(q_{|I_k^{\mathcal X}}\right)^*CH^{m-i}(I_r^G)\right),
    \]
where $m$ is the relative dimension of $pr_1: I_r\to X$ that can be calculated as follows. Over a point $x\in X$, the fiber of $pr_1: I_r\to X$ is the set of points $x'\in X$ such that { the intersection $P_x\cap P_{x'}$ contains a $\mathbb P^{r-1}$ in $P_x$}. {The set of $\mathbb P^{r-1}$ contained in $P_x$ is a $\mathbb P^r$. Furthermore, for each given $\mathbb P^{r-1}\subset Y$, the set of $r$-spaces $P_{x'}\subset \mathbb P^n$ that contain the given $\mathbb P^{r-1}$ is a $\mathbb P^{n-r}$}, and the condition that $P_{x'}\subset Y$ is equivalent to saying that the defining equation of the residual quadric is {identically} zero, which gives $\frac{(r+1)(r+2)}{2}$ independent conditions. Taking everything into account, we find that $m=n-\frac{(r+1)(r+2)}{2}=r+1$.

    Now since $\Gamma_\Psi$ is a fiber of $p\circ i: \Gamma_{\Psi_{\mathrm{univ}}}\to B$, we conclude that 
    \[ \Gamma_\Psi=\Gamma_{\Psi_{\mathrm{univ}}|{X\times X}}\in (i_{r+1,*}^\mathcal Xq^*CH^0(\Delta_G) + i_{r*}^\mathcal X q^*CH^{r+1}(I_r^G))|_{X\times X}.
    \]
    Therefore, we can write 
    \begin{equation}\label{EqDecompOfGammaPsi}
        \Gamma_\Psi=\alpha \Delta_X+\delta,
    \end{equation} 
    where $\alpha\in \mathbb Z$ is a coefficient and $\delta\in (i_{{r}*}^\mathcal X q^*CH^{r+1}(I_r^G))|_{X\times X}$. {Write $\delta= i_{r*}^\mathcal X (q^*\delta_G)$ for some $\delta_G\in CH^{r+1}(I_r^G)$}. Notice that $I_r^G$ is a fiber bundle over $\mathrm{Gr}(r+1, n+1)$ whose fiber is a closed Schubert subvariety $\Sigma$ of $\mathrm{Gr}(r+1, n+1)$, that is defined as the closure of the set of $r$-spaces in $\mathbb P^n$ that intersects with a given $r$-space along a $(r-1)$-space. This fiber bundle has a universal cellular decomposition into affine bundles in the sense of~\cite[Example 1.9.1]{Fulton}. By~\cite[Example 19.1.11]{Fulton} and~\cite[Théorème 7.33]{Voisin}, we have
    \[
    CH^*(I_r^G)=CH^*(\mathrm{Gr}(r+1, n+1))\otimes CH^*(\Sigma).
    \]
    Therefore, we can write 
    \[
    \delta_G=\sum_{i=0}^{r+1} (p_1^*\alpha_i).\beta_i,
    \]
    where $\alpha_i\in CH^i(\mathrm{Gr}(r+1, n+1)$ and $\beta_i\in CH^{r+1-i}(\Sigma)$. {One easily checks that the morphism $pr_2^*: CH^*(Gr(r+1, n+1))\to CH^*(\Sigma)$ is surjective.} Hence, there exists $\gamma_i\in CH^{r+1-i}(\mathrm{Gr}(r+1, n+1))$ such that $\beta_i=pr_2^*\gamma_i$. Therefore, 
    {
    \begin{equation}\label{EqDecompOfDeltaG}
    \delta_G=\sum_{i=0}^{r+1}p_1^*\alpha_i.p_2^*\gamma_i.
    \end{equation}
    }
    {Now we prove
    \begin{lemma}\label{LmmCoeffAlpha}
        The coefficient $\alpha$ in the decomposition (\ref{EqDecompOfGammaPsi}) equals $(-2)^{r+1}$ .
    \end{lemma}
    }
    \begin{proof}
         For $\omega\in H^{N,0}(X)$, we have $\delta^*\omega=p_{2*}(\sum_i p_1^*([\alpha_i]\cup \omega)\cup \gamma_{i|X}\cup [I_r])$. Since $\omega$ is a top degree form,  $[\alpha_i]\cup \omega=0$ unless $i=0$ and thus, $\delta^*\omega=c[\gamma_{0|X}]\cup I_r^*\omega$, where $c$ is some constant number. Now let us prove that $I_r^*\omega=0$. Let $\mathcal P_{r,r-1}=\{(x,
\lambda)\in F_r(Y)\times F_{r-1}(Y): P_\lambda\subset P_x\}$ be the flag variety. Since $I_r=\mathcal P_{r,r-1}^t\circ \mathcal P_{r,r-1}$ as {correspondences}, we have $I_r^*\omega=\mathcal P_{r,r-1}^*(\mathcal P_{r,r-1*}(\omega))=0$ since $\mathcal P_{r,r-1*}\omega\in H^{N-1,-1}(F_1(Y))$. Therefore, $\Psi^*\omega = \alpha\omega + \delta^*\omega = \alpha\omega$. By Theorem A, we get $\alpha=(-2)^{r+1}$.
    \end{proof}
    Since $z\in CH_0(X)_{\mathbb Q}$ is a {$0$-cycle}, $\delta_*z=p_{2*}(\sum_i p_1^*(\alpha_i\cdot z)\cdot p_2^*\gamma_{i|X}\cdot I_r)=p_{2*}(p_1^*(\alpha_0\cdot z)\cdot p_2^*\gamma_{0|X}\cdot I_r)=\gamma\cdot I_{r*}z$ where $\gamma$ is some multiple of $\gamma_{0|X}$, which is an element in the image of the restriction map $CH^{r+1}(\mathrm{Gr}(r+1,n+1))\to CH^{r+1}(X)$. {Taking Lemma~\ref{LmmCoeffAlpha}, formula (\ref{EqDecompOfGammaPsi}) and formula (\ref{EqDecompOfDeltaG}) into account, we prove the formula 
    \begin{equation}\label{EqFormulaForPsiz}
        \Psi_*z=(-2)^{r+1}z+\gamma\cdot I_{r*}z,
    \end{equation}
    as desired.}
\end{proof}
{
The cycle gamma appearing in (\ref{EqFormulaForPsiz}) is a polynomial in the Chern classes $c_i$ of the tautological bundle of the Grassmannian, restricted to $X$. Let us now prove}
\begin{proposition}\label{PropVanishingOfc3}
    We have $c_{r+1}\cdot I_{r*}z$=0 for $z\in CH_0(X)_{hom}$.
\end{proposition}
\begin{proof}
    Let $H_{n-1}\subset \mathbb P^n$ be a general hyperplane. Let us define the Schubert variety { $\Sigma_{r+1}^{H_{n-1}}$ as}
    \[
    \Sigma_{r+1}^{H_{n-1}}:=\{y\in \mathrm{Gr}(r+1,n+1): P_y\subset H_{n-1}\}.
    \]
    Then $c_{r+1}$ is represented by the class of $\Sigma_{r+1}^{H_{n-1}}$. Let $x\in X$ be a general point and let 
    \[
    \Theta_{x}:=\{y\in X: \dim (P_y\cap P_x)\geq r-1
    \}.
    \]
    Then {by the definition of $I_r$, }$\Theta_x$ represents $I_{r *}x$. {By definition, the intersection  $\Sigma_{r+1}^{H_{n-1}} \cap \Theta_x$ is the set
    $$\{y\in X: P_y \subset H_{n-1}\cap Y, \,{\rm dim}\,P_x\cap P_y\geq r-1\}, $$
    which we can also rewrite, if $P_x$ is not contained in $H_{n-1}$, as
    $$\Sigma_{r+1}^{H_{n-1}} \cap \Theta_x=\{y\in X:\, Y\cap H_{n-1}\supset P_y\supset P_x\cap H_{n-1}\}.$$}
    Let $\Delta_x=P_x\cap H_{n-1}$. {Then $\Delta_x$ provides a point $\delta_x$ of the variety $F_{r-1}(Y\cap H_{n-1})$}. Let $i: F_r(Y\cap H_{n-1})\hookrightarrow X$ be the natural embedding map. Let $\mathcal P_{r,r-1}^{H_{n-1}}:=\{(y, \lambda)\in F_r(Y\cap H_{n-1})\times F_{r-1}(Y\cap H_{n-1}): P_\lambda\subset P_y\}$ be the {incidence} variety. By the above description of $ \Sigma_{r+1}^{H_{n-1}}\cap \Theta_x$, one finds that 
    \[
    c_{r+1}\cdot I_{r*}x=i_*(\mathcal P_{r,r-1}^{H_{n-1}*}({\delta_x})),
    \]
    where $\Delta_x$ is viewed as an element in $CH_0(F_{r-1}(Y\cap H_{n-1}))$. One can verify that $F_{r-1}(Y\cap H_{n-1})$ is a Fano manifold and thus, $CH_0(F_{r-1}(Y\cap H_{n-1})) = \mathbb Z$. Therefore, the Chow class of $c_{r+1}\cdot I_{r*}x$ does not depend on $x$. That is, for any $z\in CH_0(X)_{hom}$, we get $c_{r+1}\cdot I_{r*}z=0$ in $CH_0(X)$, as desired. 
\end{proof}

\begin{proposition}\label{PropVanishingOfc2}
    If $CH_1(F_{r-1}(Y))_{\mathbb Q}$ is trivial, then $c_r\cdot I_{r*}(z)=0 $ in $CH_1(X)$ for any $z\in CH_0(X)_{hom}$.
\end{proposition}

\begin{proof}
    We only write down the case $r=2$, {for which the assumption of Proposition~\ref{PropVanishingOfc2} will be proved in the next section}. The general case is similar. Let $H_7$ be a $7$-dimensional linear subspace in $\mathbb P^9$. In the Grassmannian $\mathrm{Gr}(3,10)$, the class $c_2$ is represented by the subvariety $\Sigma_2^{H_7}$ defined as 
    \[
    \Sigma_2^{H_7}:=\{y\in \mathrm{Gr}(3,10): P_y\cap H_7 \textrm{ contains a line }{ \Delta_{y,H}}\}.
    \]
    For general $x\in X$, the plane $P_x$ intersects $H_7$ at a single point {$z_x$}, and the class $I_{2,*}x$ is represented by the following subvariety $\Theta_x$ in $X$ defined as
    \[
    \Theta_x:=\{y\in X: P_y\cap P_x \textrm{ contains a line }{\Delta_{xy}}\}.
    \]
    Let $H_6\subset H_7$ be a linear subspace of dimension $6$ not containing the point ${z_x}$.
    \begin{lemma}\label{LmmEquivOfC2I2z}
        { For general $x\in X$, and any $y\in \Theta_x$, we have $z_x\in \Delta_{xy}$.}
    \end{lemma}
    \begin{proof}
        Let $P_y$ be a plane in $Y$ that intersects $P_x$ along the line {$\Delta_{xy}$} and intersects $H_7$ along the line {$\Delta_{y, H}$}. Then {$\Delta_{xy}$} and {$\Delta_{y, H}$} must intersect, since they are two lines in a projective plane; and the intersection point of {$\Delta_{xy}$} and {$\Delta_{y, H}$} must be ${z_x}$ since ${z_x}$ is the only intersection point of $P_x$ and $H_7$. {Therefore, $z_x\in \Delta_{xy}$, as desired.} 
    \end{proof}
        { Let 
    $\Sigma^{H_6}_1:=\{y\in \mathrm{Gr}(3,10): P_y\cap H_6\neq \emptyset \}.$
    Let $\Xi_x$ be the variety of points $y\in X$ such that $P_y\cap P_x$  contains a line $\Delta_{xy}$ containing ${z_x}$. By Lemma~\ref{LmmEquivOfC2I2z}, we now conclude that
    \begin{equation}\label{EqEquivOfC2I2z}
        \Sigma_2^{H_7}\cap \Theta_x
 = \Sigma_1^{H_6} \cap \Xi_x.
    \end{equation}
    Indeed, by Lemma~\ref{LmmEquivOfC2I2z}, $\Sigma_2^{H_7}\cap \Theta_x$ contains all the points $y\in X$ such that $P_y\cap P_x$ contains a line $\Delta_{xy}$ containing ${z_x}$  and such that $P_y\cap H_7$ contains a line $\Delta_{y, H}$. This variety coincides with $\Sigma_1^{H_6}\cap \Xi_x$ since knowing that ${z_x}\in P_y$ and that $z_x\not\in H_6$, $P_y\cap H_7$ contains a line if and only if $P_y\cap H_6$ is nonempty.
    }
    
    Let $\mathcal P_{2,1}:=\{(x', [l])\in F_2(Y)\times F_1(Y): l\subset P_{x'}\}$ be the flag variety. Let $\Delta_{P_x, {z_x}}^*=\{[l]\in F_1(Y): {z_x}\in l\subset P_x\}$. { Then $\Delta_{P_x, z_x}^*$ provides a class $\delta_{P_x, z_x}\in CH_1(F_1(Y))$. The equation (\ref{EqEquivOfC2I2z}) shows that $c_2\cdot I_{2,*}x=c_1\cdot \mathcal P_{2,1}^*(\delta_{P_x,{z_x}}^*)$}. Therefore, for $z\in CH_0(X)_{\mathbb Q,hom}$, we have $c_2\cdot I_{2,*}z=c_1\cdot \mathcal P_{2,1}^*(Z)$ {for some} $Z\in CH_1(F_1(Y))_{\mathbb Q, hom}$. {But we have $CH_1(F_1(Y))_{\mathbb Q, hom}=0$ by Theorem~\ref{ThmChowOneOfF1} below (or by assumption for $r>2$)}. Hence, $c_2\cdot I_{2,*}z=0\in CH_1(X)_{\mathbb Q}$, as desired.
\end{proof}

\subsection{Triviality of $CH_1(F_1(Y))_{\mathbb{Q}}$}

Let $Y\subset \mathbb P^9$ be a cubic $8$-fold. It has been established that $H^{p,q}(F_1(Y)) = 0$ for $p \leq 1$ and $p \neq q$~\cite{DebarreManivel}, indicating that the coniveau of $F_1(Y)$ is at least $2$. According to the generalized Bloch conjecture, this suggests that $CH_i(F_1(Y))_{\mathbb{Q}, hom} = 0$ for $i \leq 1$. {This section is devoted to proving this statement, namely
\begin{theorem}\label{ThmChowOneOfF1}
The Chow group $CH_1(F_1(Y))_{\mathbb{Q}}$ is isomorphic to $\mathbb{Q}$.
\end{theorem}
}
\subsubsection{Lines of lines}
Let $P$ be a plane contained in $Y$ and let $x\in P$ be a point. Let $\Delta_{P,x}^*$ be the variety of lines in $P$ passing through $x$. Viewed as a $1$-cycle of $F_1(Y)$, it is clear that the Chow class of $\Delta_{P,x}^*$ does not depend on the choice of $x\in P$ and we let $\Delta_P^*\in CH_1(F_1(Y))$ denote the Chow class of $\Delta_{P,x}^*$ for some (and thus any) $x\in P$.

\begin{proposition}\label{PropTrianglePlanesConstant}
    Let $L\cong\mathbb P^3\subset \mathbb P^9$ be a {$3$}-dimensional linear subspace whose intersection with $Y$ contains three planes $P_1, P_2$ and $P_3$ (not necessarily distinct). Then the Chow class
    \[
\Delta_{P_1}^*+\Delta_{P_2}^*+\Delta_{P_3}^*\in CH_1(F_1(Y))
    \]
    does not depend on the choice of $L$.
\end{proposition}
\begin{remark}
    The linear spaces $L\cong\mathbb P^3$ whose intersection with $Y$ is the union of three planes form a projective variety of general type, which is not rationally connected. To prove the proposition, we consider a larger space, namely the variety of cubic surfaces in $Y$ which are cones and prove that this space is rationally connected. 
\end{remark}

\begin{proof}[Proof of Proposition~\ref{PropTrianglePlanesConstant}]
    Let $\mathcal M$ be the space of cubic surfaces in $Y$ which are cones. An element of $\mathcal M$ is of the form $(S, y_0)$ where $S$ is a cubic surface that is a cone from a vertex $y_0$. If $P_1, P_2, P_3$ are planes as in the Proposition, and $y_0$ is a point in the intersection of the three planes, then $P_1\cup P_2\cup P_3$ is a cone with vertex $y_0$, hence $(P_1\cup P_2\cup P_3, y_0)$ is an element of $\mathcal M$. Consider the incidence variety 
\[
\mathcal E=\{((S,y_0), l)\in \mathcal M\times F_1(Y): l \textrm{ is a line in } S \textrm{ passing through } y_0\}.
\]
Then $\Delta_{P_1}^*+\Delta_{P_2}^*+\Delta_{P_3}^*$ is the $1$-cycle given by $\mathcal E_*(P_1\cup P_2\cup P_3, y_0)$. Therefore, to prove the Proposition, it suffices to show that $\mathcal M$ is rationally connected.

Let $\pi: \mathcal M\to Y$ be the rational map that sends a cubic surface $(S, y_0)$ that is a cone onto its vertex $y_0$. The fiber of $\pi$ over $y_0$ parametrizes the cubic surfaces in $Y$ that are cones with vertex $y_0$. Then $S$ is contained in the singular hyperplane section $Y_{y_0}:=Y\cap \bar T_{Y,y_0}$ where $\bar T_{Y,y_0}$ is the projective tangent space of $Y$ at the point $y_0$. We may take $y_0$ a general point so that the cubic hypersurface $Y_{y_0}$ of dimension $7$ has an ordinary double point at $y_0$, and does not contain any $\mathbb P^3$, and so that no $3$-dimensional linear subspaces in $\mathbb P^9$ passing through $y_0$ is contained in $Y$. The last condition is satisfied since the $3$-dimensional linear subspaces contained in $Y$ form a divisor in $Y$, as can be shown by a simple dimension counting argument. 

\begin{lemma}\label{LmmFiberOfPiOverGeneralPoint}
    Under the assumptions above on $y_0$, the fiber of $\pi$ over $y_0$ is in bijection with the set of planes contained in the Hessian quadric $Q_{y_0}$ of $Y_{y_0}$ at the point $y_0$.
\end{lemma}
\begin{proof}
    A cubic surface  $S=\mathbb P^3\cap Y_{y_0}$ in $Y_{y_0}$ passing through $y_0$ is a cone  with vertex $y_0$ if and only if the equation $f_{y_0}$ defining $Y_{y_0}$ has vanishing Hessian, which means that the Hessian quadric vanishes on the  tangent space $T_{\mathbb P^3, y_0}$ .
\end{proof}

One knows that the variety of planes in a quadric sixfold is a connected Fano manifold. Thus, by Lemma~\ref{LmmFiberOfPiOverGeneralPoint}, the map $\pi$ is a fibration whose base and general fiber are rationally connected. Therefore, the total space $\mathcal M$ is rationally connected \cite{GraberHarrisStarr}, as desired.
\end{proof}

\begin{definition}\label{DefOfDelta}
    Let $\Delta^*$ be an element of $CH_1(F_1(Y))_{\mathbb Q}$ defined by $\frac13(\Delta_{P_1}^*+\Delta_{P_2}^*+\Delta_{P_3}^*)$ where $P_1, P_2, P_3$ are the planes as in Proposition~\ref{PropTrianglePlanesConstant}.
\end{definition}

\begin{corollary}
    Let $L\subset Y$ be a $3$-dimensional linear space in $Y$. For any plane $P\subset L$, we have $\Delta^*_P=\Delta^*$ in $CH_1(F_1(Y))_{\mathbb Q}$.
\end{corollary}
\begin{proof}
    This is because the triple $(P,P,P)$ satisfies Proposition~\ref{PropTrianglePlanesConstant} and thus $3\Delta^*_P=3\Delta^*$, from which we conclude.
\end{proof}
\begin{corollary}\label{CorImOfFanoOf3SpacesDoesNotDependOn3Spaces}
    Let $L\subset Y$ be a $3$-dimensional linear space in $Y$. Then the image of the natural morphism 
    \[
    CH_1(F_1(L))_{\mathbb Q}\to CH_1(F_1(Y))_{\mathbb Q}
    \]
    is of dimension $1$ and is generated by $\Delta^*$. In particular, the image does not depend on the choice of $L\in F_3(Y)$.
\end{corollary}

\begin{remark}
    We will see, in Proposition~\ref{PropLinesOfLinesAreConstant}, that for any $P\subset Y$, we have $\Delta^*_P=\Delta^*$.
\end{remark}

\subsubsection{The geometry of $F_1(Y)$  and its $1$-cycles}
The following Lemma is proved by an easy dimension count.

\begin{lemma}\label{LmmP3CoverADivisor}
    Assume $Y$ is general. The linear subspaces  $\mathbb P^3\subset Y$ cover a divisor $D$  in $Y$.
\end{lemma}
\begin{corollary}\label{CorLineIntersectFiniteP3}
    A general  line $\Delta\subset Y$ meets finitely many $\mathbb P^3\subset Y$.
\end{corollary}
Indeed, if $\Delta \not \subset D$, the  $\mathbb P^3\subset Y$ intersecting $\Delta$ are in bijection with the intersection points of $\Delta$ and $D$.
Let $F_{3,1}\subset  F_3(Y)\times F_1(Y)$ be the set of pairs $(x,s)$ such that $P_x\cap \Delta_s\not=\emptyset$. The second projection $q:F_{3,1}\rightarrow F_1(Y)$ is dominant by Lemma~\ref{LmmP3CoverADivisor} and  generically finite by  Corollary~\ref{CorLineIntersectFiniteP3}.

\begin{lemma}\label{LmmMorphismF1F31}
     (i) The morphism  $q^*: CH_1(F_1(Y))\rightarrow CH_1(F_{3,1})$ is injective.\\
(ii) The morphism  $p_*\circ q^*: CH_1(F_1(Y))_{hom}\rightarrow CH_1(F_3(Y))_{hom}$ is  zero.
\end{lemma} 
\begin{proof} (i) follows from the fact that $q$ is dominant. For (ii), we observe that there is a natural birational map 
\[ \pi: \mathcal P_3\times_{Y} \mathcal P_1\to F_{3,1},\]
where $$\mathcal P_3\subset F_3(X)\times Y,\,\, \mathcal P_1\subset F_1(X)\times Y$$ are the incidence correspondences. The morphism $\pi$ is commutative with the projections to $F_3(Y)$ and to $F_1(Y)$. By direct calculations,
\begin{align*}
    \mathcal P_3^*\circ \mathcal P_{1*} & = p_*\circ \pi_* \circ \pi^* \circ q^* \\
    & = p_* \circ q^*
\end{align*}
Thus  the  morphism  $p_*\circ q^*$ factors through the morphism $(\mathcal P_1)_* :  CH_1(F_1(Y))_{hom}\rightarrow CH_2(Y)_{hom}$, which is zero because  $CH_2(Y)_{hom}=0$ by~\cite{Otwinowska}.
\end{proof}

The variety $F_{3,1}$ admits a rational map $f$  to the projective bundle $\mathcal{P}_5$ over $F_3(Y)$ whose fiber over $x\in F_3(Y)$ is the set of $\mathbb{P}^4$ containing $P_x$. To a pair $(P_x,\Delta_t)$ of a $\mathbb P^3$ and a line which intersect, this map associates $\langle P_x,\Delta_t\rangle$. We introduce now a desingularization $\tau: \widetilde{F_{3,1}} \rightarrow  F_{3,1}$ on which $f$ desingularizes to a morphism $\tilde{f}: \widetilde{F_{3,1}}\rightarrow \mathcal{P}_5$.  We observe now the following:  for each $x\in F_3(Y)$ and $4$-dimensional space $P'_x$ containing $P_x$, the intersection $P'_x\cap Y$ is the union  $P_x\cup Q_x$ where $Q_x$ is a $3$-dimensional quadric intersecting $P_x$ along a $2$-dimensional quadric.  The general fiber  of  $\tilde{f}$ over $(x, P'_x)$ is birational to the family of lines in the $3$-dimensional quadric $Q_x$. We are now going to  prove
\begin{proposition}\label{ThmF31ToF1}
    Let $z\in CH_1(\widetilde{F_{3,1}})_{\mathbb Q}$ be a $1$-cycle such that $\tilde f_*(z)=0\in CH_1(\mathcal P_5)$. Then $\tilde q_*(z)=\alpha \Delta^*\in CH_1(F_1(Y))_{\mathbb Q}$ for some $\alpha\in \mathbb Q$.
\end{proposition}
We will use for this the following  general result of Bloch-Srinivas type~\cite{VoisinCitrouille}:
\begin{lemma}\label{ThmSpreading}
    Let $f: Z\to B$ be a surjective projective morphism between algebraic varieties. Let $B^0$ be an open dense subset of $B$ such that\\
    (a) $Z-f^{-1}(B^0)$ is of codimension at least $2$, and that\\
    (b) Every fiber of $f$ over $B^0$ has trivial $CH_0$.\\
    Let $z\in CH_1(Z)_{\mathbb Q}$ be a $1$-cycle in $Z$ such that $f_*(z)=0\in CH_1(B)_{\mathbb Q}$. Then $z$ is supported on the fibers of $f$ over $B^0$. More precisely, there are points $b_1, \ldots, b_r\in B^0$, such that $z$ is $\mathbb Q$-rationally equivalent to a $1$-cycle supported on $f^{-1}(b_1)\cup\ldots\cup f^{-1}(b_r)$.
\end{lemma}

\begin{proof}
    The case when $\dim B=1$ is rather trivial, so let us assume $\dim B\geq 2$. Let $W\subset Z$ be a multisection of degree $N$ of $f:Z\to B$. Let $f_W: W\to B$ be the restriction of $f$ to $W$. Since $Z-f^{-1}(B^0)$ is of codimension at least $2$, by Chow moving lemma, we may assume that $z$ is supported on $f^{-1}(B^0)$. We may also assume that none of the components of $z$ lies entirely in the branched locus of $f_W$. Write $z=z_1+z_2+\ldots+z_s-z_{s+1}-\ldots -z_t$ with $z_i$ irreducible curves in $f^{-1}(B^0)$. Consider $B_{z_i}=f(z_i)\subset B^0$. We may assume $B_{z_i}$ is dimension $1$ since otherwise the component $z_i$ is supported on the fibers of $f$. Let $f_i: f^{-1}(B_{z_i})\to B_{z_i}$ be the restriction of $f$. Since $\dim B\geq 2$, we may further assume, by Chow moving lemma, that the map $f_i: f^{-1}(B_{z_i})\to B_{z_i}$ restricted to $z_i$ is a birational map. Since every fiber of $f$ over $B^0$ has trivial $CH_0$, the $1$-cycle $z_i-\frac1N f^{-1}_W(B_{z_i})$ in $CH_1(f^{-1}(B_{z_i}))$ restricts to $0$ for the general fiber of $f_i: f^{-1}(B_{z_i})\to B_{z_i}$, so by the Bloch-Srinivas construction \cite{Voisin}, the cycle $z_i-\frac1N f^{-1}_W(B_{z_i})$ is supported on the fibers of $f_i$. Summing up the components, we find that the $1$-cycle
    $z-\frac1N f^*_Wf_*(z)\in CH_1(Z)_{\mathbb Q}
    $
    is supported on the fibers of $f$ over $\bigcup_i B_{z_i}\subset B^0$. However, $f_*(z)=0$ by assumption. Hence, $z\in CH_1(Z)_{\mathbb Q}$ is supported on the fibers of $f$ over $B^0$.
\end{proof}
We are going to apply Lemma~\ref{ThmSpreading} to $\tilde f: \widetilde{F_{3,1}}\to \mathcal P_5$ and to the following open set $\mathcal P_5^0$ defined as the set parametrizing the pairs $(L_3, P_4)$ such that \\
    (a) either $Q$ is smooth,\\
    (b) or $Q$ is singular at only one point $y$ and $L$ does not contain $y$.\\
It is clear that the fibers of $\tilde f$ over $\mathcal P^0_5$ are $CH_0$ trivial. Our first step is thus to check assumption (a) in Lemma~\ref{ThmSpreading}. We prove by a case by case analysis the following
\begin{lemma}\label{LmmCodim}
    $\tilde f^{-1}(\mathcal P_5^0)\subset \widetilde{F_{3,1}}$ has complement of codimension $\geq 2$.
\end{lemma}
\begin{proof}
    The complement $R$ of $\mathcal P^0_5$ in $\mathcal P_5$ is stratified by the following subsets $R_4, R_3, R_2, R_1$, where $R_i$ parametrizes $(L, P_4)\in R$ such that the residual $Q$ is of rank $i$.
        
         \textbf{Analysis of $R_4$.} The stratum $R_4$ parametrizes $(L, P_4)\in \mathcal P_5$ such that  $Q$ is singular at one point (equivalently, the rank of $Q$ is $4$) and that $L$ contains the singular point of $Q$. In this case, $Q$ is a cone from its singular point over a smooth quadric surface. 
        \begin{sublemma}\label{LmmLocusOfSimpleSingularQuadric}
            For any $3$-dimensional linear subspace $L$ contained in $Y$, the set of $4$-dimensional subspaces $P_4\subset \mathbb P^9$ containing $L$ such that $Q$ has a single singular point lying on $L$ has codimension at least $2$ in $\mathbb P^5=$\{$4$-dimensional subspaces $P_4\subset \mathbb P^9$ containing $L$\}.
        \end{sublemma}
        \begin{proof}
            Without loss of generality, let us assume $L=\{(x_0,\ldots, x_3, 0,\ldots, 0)\}\subset \mathbb P^9$. The fact that $L\subset Y$ implies that the defining equation of $Y$ is of the form
            \[
            f(x_0,\ldots, x_9)=\sum_{i=4}^9 x_iQ_i(x_0, \ldots, x_9),
            \]
            where $Q_i$ is a quadratic polynomial in the variables $x_0, \ldots, x_9$, for each $i\in \{4,\ldots, 9\}$. Let $\mathbf a:=(a_0, a_1, \ldots, a_5)\in \mathbb P^5$. Then $\mathbf a$ determines a dimension $4$ linear subspace $P_4$ containing $L$ as 
            \[
            P_4=\{(x_0, x_1, x_2, x_3, ta_0, ta_1,\ldots, ta_5): (x_0, x_1, x_2, x_3, t)\in \mathbb P^4\}.
            \]
            The corresponded residual quadric hypersuface $Q_{\mathbf a}$ is thus defined by 
            \[
            Q_{\mathbf a}(x_0, x_1, x_2, x_3, t)=\sum_{i=4}^9 a_{i-4} Q_i(x_0, x_1, x_2, x_3, ta_0, ta_1,\ldots, ta_5).
            \]
            We can identify the quadratic form $Q_{\mathbf a}(x_0, x_1, x_2, x_3, t)$ with a $5\times 5$ symmetric matrix that we still denote as $Q_{\mathbf a}$. Viewed as a function of $\mathbf a\in \mathbb P^5$, the matrix $Q_{\mathbf a}$ is a section of $\mathrm{Sym}^2(\mathcal O(1)_{\mathbb P^5}\oplus \mathcal O_{\mathbb P^5}^4)\otimes \mathcal O_{\mathbb P^5}(1)$, and thus the locus of $\mathbf a\in\mathbb P^5$ where $Q_{\mathbf a}$ degenerates is a degree $7$ hypersurface in $\mathbb P^5$. On the other hand, the quadratic polynomial of $Q_{\mathbf a}\cap L$ is $Q_{\mathbf a}(x_0, x_1, x_2, x_3, t=0)$, which, viewed as a $4\times 4$ symmetric matrix varing with $\mathbf a$, is a section of $\mathrm{Sym}^2( \mathcal O_{\mathbb P^5}^4)\otimes \mathcal O_{\mathbb P^5}(1)$, so that the degenerate locus of  $Q_{\mathbf a}\cap L$ is of degree $4$. Therefore, there exists a point $\mathbf a\in \mathbb P^5$ such that $Q_{\mathbf a}$ is singular whereas $Q_{\mathbf a}\cap L$ is smooth. Therefore, the locus of $\mathbf a\in\mathbb P^5$ such that $Q_{\mathbf a}$ is singular at exactly one point and that this point lies in $L$ is strictly contained, as a closed subset, in the locus where $Q_{\mathbf a}$ is singular at only one point. The latter locus is of codimension $1$ in $\mathbb P^5$ since it is defined by the vanishing of the determinant of the matrix $Q_{\mathbf a}$. Hence, the set of $4$-dimensional subspaces $P_4\subset \mathbb P^9$ containing $L$ such that $Q$ has a single singular point lying on $L$ has codimension at least $2$ in $\mathbb P^5$, as desired.
        \end{proof}
        \begin{sublemma}\label{LmmFiberOverASimpleSingularQuadric}
            The fiber of $\tilde f: \widetilde{F_{3,1}}\to \mathcal P_5$ over an element $(L,P_4)$ such that the corresponding residual quadric $Q$ has only one singularities is isomorphic to the union 
            \[
            \mathcal P_2^*\cup \mathcal P_2^*,
            \]
            where $\mathcal P_2^*$ is a $\mathbb P^2$-bundle over $\mathbb P^1$, and the two $\mathcal P_2^*$ intersect at a smooth quadric surface.
        \end{sublemma}
        \begin{proof}
            The singular quadric hypersurface $Q$ is a cone from its only singular point $y$ over a smooth quadric surface $Q'$. Each line in $Q'$, together with $x$ determines a plane, and every line in $Q$ lies in some of these planes. There are two pencils of lines in $Q'$, each of which induces a pencil of planes in $Q'$. Each plane contains a $\mathbb P^2$ of lines. Hence, the lines in $Q$ is the union of two $\mathbb P^2$-bundles over $\mathbb P^1$. The intersection of these $\mathbb P^2$-bundles is the set of lines in $Q$ that pass through the singular point $y$, which is in turn isomorphic to $Q'$. 
        \end{proof}
        Sublemma~\ref{LmmLocusOfSimpleSingularQuadric} implies that $R_4\subset \mathcal P_5$ is of codimension at least $2$. Sublemma~\ref{LmmFiberOverASimpleSingularQuadric} shows that the dimension of the fiber of $\tilde f$ over $R_4$ is $3$, the same with that of the general fiber of $\tilde f: \widetilde{F_{3,1}}\to \mathcal P_5$. Hence, the codimension of $\tilde f^{-1}(R_4)$ is at least $2$. This complete the analysis of $R_4$.
        
        \textbf{Analysis of $R_3$.} The stratum $R_3$ parametrizes $(L, P_4)\in \mathcal P_5$ such that $Q$ is of rank $3$. In this case, $Q$ is singular along a line and is a cone from one of its singular point over a quadric cone surface.
        \begin{sublemma}\label{LmmCodimOfStrata}
            The stratum $R_3, R_2, R_1$ is of codimension $3, 6, 10$ in $\mathcal P_5$. In particular, $R_1=\emptyset$.
        \end{sublemma}
        \begin{proof}
            Let $\mathcal E_4$ and $\mathcal Q_6$ be the tautological subbundle, respectively, tautological quotient bundle over $F_3(Y)$. Then $\mathcal P_5$ is nothing but the projectivization of $\mathcal Q_6$. Let $\mathcal H$ be the Hopf bundle of $\mathbb P(\mathcal Q_6)$, which is a subbundle of $\pi^*\mathcal Q_6$. Let $\mathcal F_5$ be the kernel of the composite of the canonical maps $V_{10}\to \pi_*\mathcal Q_6\to \pi_*\mathcal Q_6/\mathcal H$, where the first map is the universal quotient map of bundles of $\mathbb P(\mathcal Q_6)$. Then there is a natural exact sequence of vector bundles on $\mathbb P(\mathcal Q_6)$:
            \[
            0\to \pi^*\mathcal E_4\to \mathcal F_5\to \mathcal H\to 0.
            \]
            The defining equation of $Y\subset \mathbb P^9$ gives a section of the bundle $\mathrm{Sym}^3\mathcal F_5^*$ whose image in $\pi^*\mathrm{Sym}^3\mathcal E_4^*$ vanishes since the fibers of $\mathbb P(\mathcal E_4)$ are by definition $3$-dimensional linear subspaces in $Y$. Hence, the universal residual quadric hypersurface is defined by a section of the bundle $\mathrm{Sym}^2\mathcal F_5^*\otimes \mathcal H^*$. Notice that $\mathrm{Sym}^2\mathcal F_5^*\otimes \mathcal H^*$ is generated by global sections. Since $Y$ is general, the locus where such a quadratic form is of rank $\leq r$ is $\binom{6-r}{2}$, for $r\in\{0,\ldots, 4\}$, as desired.
        \end{proof}
        \begin{sublemma}\label{LmmFiberOverR3}
            The fiber of $\tilde f: \widetilde{F_{3,1}}\to \mathcal P_5$ over a point $(L, P_4)$ in $R_3$ is of dimension $3$.
        \end{sublemma}
        \begin{proof}
            The residual quadric hypersurface $Q$ is singular along a line and is a cone from one of its singular point over a quadric cone surface. By the same argument as in Sublemma~\ref{LmmFiberOverASimpleSingularQuadric}, the variety of lines in $Q$ is isomorphic to a $\mathbb P^2$-bundle over a smooth conic, thus has dimension $3$.
        \end{proof}
        By Sublemma~\ref{LmmCodimOfStrata} and Sublemma~\ref{LmmFiberOverR3}, the codimension of $\tilde f^{-1}(R_3)$ is $3$. This complete the analysis of $R_3$.
        
        \textbf{Analysis of $R_2$.} The stratum $R_2$ parametrizes $(L, P_4)\in \mathcal P_5$ such that $Q$ is of rank $2$. In this case, $Q$ is a union of two $3$-dimensional linear subspaces intersecting along a plane. It is clear that in this case, the variety of lines in $Q$ is isomorphic to the union of two $\mathrm{Gr}(2,4)$ intersecting along a $\mathbb P^2$. Hence, the dimension of fibers of $\tilde f$ over $R_2$ is $4$. By Sublemma~\ref{LmmCodimOfStrata}, the codimension of $\tilde f^{-1}(R_2)$ is $5$. This complete the analysis of $R_2$.
        
        \textbf{Analysis of $R_1$.} The stratum $R_1$ parametrizes $(L, P_4)\in \mathcal P_5$ such that $Q$ is of rank $1$. By Sublemma~\ref{LmmCodimOfStrata}, this case does not happen.
        
        The case where $Q$ is of rank $0$ does not happen, since in this case, $P_4$ is contained in $Y$, which cannot happen if $Y$ is a general cubic $8$-fold.
\end{proof}

\begin{lemma}\label{LmmSurjectivity}
    For any $(P, P_4)$ in $\mathcal P_5^0$, with associated $3$-dimensional quadric $Q$, the natural morphism 
\[
CH_1(F_1(P\cap Q))_{\mathbb Q}\to CH_1(F_1(Q))_{\mathbb Q}
\]
is surjective.
\end{lemma}

\begin{proof}
\textbf{Case (a).} In this case, $CH_1(F_1(Q))_{\mathbb Q}=\mathbb Q$ since $F_1(Q)$ is a connected Fano manifold. Therefore, $CH_1(F_1(L\cap Q))_{\mathbb Q}\to CH_1(F_1(Q))_{\mathbb Q}$ has to be surjective. 

\textbf{Case (b).} In this case, $Q$ is a cone from its singular point $y$ over a smooth quadric surface $Q'$. By Lemma~\ref{LmmFiberOverASimpleSingularQuadric}, the Fano variety of lines $F_1(Q)$ is isomorphic to $\mathcal P_2^*\cup \mathcal P_2^*$ where $\mathcal P_2^*$ is a $\mathbb P^2$-bundle over $\mathbb P^1$ and the two $\mathbb P_2^*$ intersect along a smooth quadric surface $Q'$. The variety $F_1(L\cap Q)$ is a disjoint union of two lines $\ell_1$ and $\ell_2$, each representing a pencil of lines on the smooth quadric surface $L\cap Q$. Hence, $CH_1(F_1(L\cap Q))$ is generated freely by two cycles $z_1$ and $z_2$ where $z_i$ represents $\ell_i$ for $i=1,2$. Let us analyze $CH_1(F_1(Q))$. We know that $F_1(Q)=\mathcal P_2^*\cup \mathcal P_2^*$ where each $\mathcal P_2^*$ is a $\mathbb P^2$-bundle over $\ell_i$ respectively. Let us consider the first $\mathcal P_2^*$ which is a $\mathbb P^2$-bundle over $\ell_1$. Let $p_1: \mathcal P_2^*\to \ell_1$ be the corresponding map, and let $h$ be the auti-tautological class of this projective bundle. The Chow group $CH_1(\mathcal P_2^*)$ is then generated by $h^2\cdot p_1^*(\ell_1)$ and $h\cdot p_1^*(pt)$. The image of the class $h^2\cdot p_1^*(\ell_1)$ in $CH_1(F_1(Q))$ is the same as the class of $z_1$ in $CH_1(F_1(Q))$ and the image of the class $h\cdot p_1^*(pt)$ is the as the class of $z_2$ in $CH_1(F_1(Q))$. Hence, the image of $CH_1(F_1(L\cap Q))\to CH_1(F_1(Q))$ contains the image of $CH_1(\mathcal P_2^*)\to CH_1(F_1(Q))$ for the first $\mathcal P_2^*$. Similarly, the image of $CH_1(F_1(L\cap Q))\to CH_1(F_1(Q))$ contains the image of $CH_1(\mathcal P_2^*)\to CH_1(F_1(Q))$ for the second $\mathcal P_2^*$. Therefore, $CH_1(F_1(L\cap Q))\to CH_1(F_1(Q))$ is surjective.
\end{proof}

We finally conclude the proof of Proposition~\ref{ThmF31ToF1}.
\begin{proof}[Proof of Proposition~\ref{ThmF31ToF1}]
    Let $z\in CH_1(\widetilde{F_{3,1}})_{\mathbb Q}$ be a $1$-cycle such that $\tilde f_*(z)=0\in CH_1(\mathcal P_5)_{\mathbb Q}$. Lemma~\ref{LmmCodim} and the discussion above shows that Lemma~\ref{ThmSpreading} can be applied to the map $\tilde f:\widetilde{F_{3,1}}\to \mathcal P_5$ with $\mathcal P_5^0$ the open dense subset, so that we conclude that $z$ is supported on fibers of $\tilde f$ over $\mathcal P_5^0$. Write $z=z_1+\ldots +z_r\in CH_1(\widetilde{F_{3,1}})_{\mathbb Q}$ where $z_i$ is supported on $F_1(Q_i)$ with $Q_i$ is a residual quadric hypersurface coming from a pair $(L_i, P_{4,i})$ in $\mathcal P_5^0$. By Lemma~\ref{LmmSurjectivity}, for each $i$, the the natural morphism $CH_1(F_1(Q_i\cap L_i))_{\mathbb Q}\to CH_1(F_1(Q_i))_{\mathbb Q}$ is surjective. Hence, $\tilde q_*(z)$ lies in the sum of the images of $CH_1(F_1(L_i))_{\mathbb Q}\to CH_1(F_1(Y))_{\mathbb Q}$. By Corollary~\ref{CorImOfFanoOf3SpacesDoesNotDependOn3Spaces}, the cycle $\tilde q_*(z)$ is a multiple of $\Delta^*$, as desired.
\end{proof}

We next prove that all $\Delta_P^*$ has the same Chow class in $CH_1(F_1(Y))_{\mathbb Q}$.

\begin{proposition}\label{PropLinesOfLinesAreConstant}
    Let $P\subset Y$ be a plane. Then $\Delta_P^*=\Delta^*$ in $CH_1(F_1(Y))_{\mathbb Q}$.
\end{proposition}

\begin{proof}
    We may assume $P\subset Y$ is general. Let $P'\subset Y$ be another general plane intersecting $P$ along a line $l$. By Corollary~\ref{CorLineIntersectFiniteP3}, there is a $3$-dimensional linear space $L\subset Y$ such that $l$ intersects $L$ at a point $y$. As $P$ and $P'$ are general, we may assume that both $P$ and $P'$ intersect $L$ at only one point $y$. The line of lines in $P$ passing through $y$ naturally lifts to a curve $Z_1\subset \widetilde{F_{3,1}}$ contained in the fiber of $\pi\circ\tilde f$ over the point $l_3$ of $F_3(Y)$ parameterizing $L_3$. We consider the $1$-cycle $z:=[Z_1] - [Z_1']\in CH_1(\widetilde{F_{3,1}})_{\mathbb Q}$. It is clear that $\tilde f_*(z)=0$ in $CH_1(\mathcal P_5)_{\mathbb Q}$ since $\tilde f_*([Z_1])$ and $\tilde f_*([Z_1'])$ is represented by two lines in the fiber of the $\mathbb P^5$-bundle $\mathcal P_5$ over the point $L\in F_3(Y)$. By Proposition~\ref{ThmF31ToF1}, we have $\tilde q_*(z)=\alpha \Delta^*$ in $CH_1(F_1(Y))_{\mathbb Q}$. But it is clear that $\tilde q_*([Z_1])=\Delta_P^*$ and $\tilde q_*([Z_1'])=\Delta_{P'}^*$, and that $\alpha=0$ by degree reasons. Hence, $\Delta_P^*=\Delta_{P'}^*$. There exists a $P_3\cong\mathbb P^3$ containing $P$ and $P'$ and the intersection $P_3\cap Y$ is the union of three planes $P, P', P''$. The same argument as above shows that $\Delta_P^*=\Delta_{P'}^*=\Delta_{P''}^*$. Finally, Proposition~\ref{PropTrianglePlanesConstant} shows that $\Delta_P^*=\Delta^*$, as desired.
\end{proof}

We now conclude the proof of Theorem~\ref{ThmChowOneOfF1} in this section. 
\begin{proof}[Proof of Theorem~\ref{ThmChowOneOfF1}]
It is not hard to show that $H^2(F_1(Y),\mathbb Q)=\mathbb Q$ since the restriction map $H^2(\mathrm{Gr}(2,10),\mathbb Q)\to H^2(F_1(Y),\mathbb Q)$ is an isomorphism \cite{DebarreManivel}. It suffices to prove that $CH_1(F_1(Y))_{\mathbb Q, hom}=0$. Let $\alpha\in CH_1(F_1(Y))_{\mathbb Q, hom}$ and let $z=\tilde q^*\alpha\in CH_1(\widetilde{F_{3,1}})_{\mathbb Q, hom}$. Since $\tilde q_*(z)=\deg \tilde q\cdot \alpha$, it suffices to prove that $\tilde q_*(z)=0$. 

Since $\mathcal P_5$ is a $\mathbb P^5$-bundle over $F_3(Y)$, we have $CH_1(\mathcal P_5)=h^5\cdot \pi^*CH_1(F_3(Y))\oplus h^4\cdot \pi^*CH_0(F_3(Y))$. Lemma~\ref{LmmMorphismF1F31} shows that $\tilde p_*z = 0 \in CH_1(F_3(Y))_{\mathbb Q, hom}$. Hence, $\tilde f_*(z)\in h^4\cdot \pi^*CH_0(F_3(Y))_{\mathbb Q, hom}$. Write $\tilde f_*(z)=w_1+\ldots+w_s$ where $w_i$ is a $1$-cycle supported on the fiber of $\pi$ over a point $L_i\in F_3(Y)$. Let $P_i$ be a plane in $Y$ that intersects with $L_i$ at only one point $y_i$. Let $P_{5,i}$ be the $5$-dimensional linear subspace spanned by $L_i$ and $P_i$. Let $z_i$ be the class in $CH_1(\widetilde{F_{3,1}})$ represented by the variety $Z_i:=\{(L_i, l): y_i\in l\subset P_i\}$. Since the fibers of $\pi: \mathcal P_5\to F_3(Y)$ are projective spaces $\mathbb P^5$, the cycle $w_i$ is proportional to the class represented by the variety $\{(L_i, P_4)\in \mathcal P_5: P_4\subset P_{5,i}\}$, which is the image of $Z_i$ under $\tilde f$. Hence, with an appropriate choice of coefficients $a_i\in\mathbb Q$, we have \[
\tilde f_*(z)=\sum_i w_i=\sum_i a_i \tilde f_*z_i=\tilde f_*(\sum_i a_iz_i).
\] 
By Proposition~\ref{ThmF31ToF1}, we conclude that 
\[\tilde q_*(z)=\tilde q_*(\sum_i a_iz_i)=\sum_i a_i\Delta_{P_i}^*=a \Delta^*,
\]
where $a=\sum_ia_i$. The last equality is due to Proposition~\ref{PropLinesOfLinesAreConstant}. Since $\tilde q_*(z)$ is homologue to $0$, the coefficient $a=0$. Hence, $\tilde q_*(z)=0$, as desired. This terminates the proof of Theorem~\ref{ThmChowOneOfF1}.
\end{proof}

\subsection{Proof of Theorem B}
We prove in this section Theorem B from the introductiono. Putting together Theorem~\ref{ThmDecOfActionOfPsiGeneralCase}, Proposition~\ref{PropVanishingOfc3}, Proposition~\ref{PropVanishingOfc2} and Theorem~\ref{ThmChowOneOfF1}, we conclude that formula (\ref{EqDecOfTheActionOfPsi}) becomes 
\begin{equation}
    \Psi_*z = -8z + \gamma' I_2^*z
\end{equation}
for any cycle $\gamma' = ac_1^3 + b'c_1c_2 + c'c_3$ on $X$, where the number $a$ is determined by the class $\gamma$ of (\ref{EqDecOfTheActionOfPsi}) by $\gamma = ac_1^3 + bc_1c_2 + cc_3$. We take for $\gamma'$ a multiplee of the class of the fixed locus of $F$ of $\Psi$. Indeed, Proposition~\ref{ThmChowClassOfF} proved in Section~\ref{SectionChowClassOfFInX} says that the class of $F$ in $\mathrm{CH}^3(X)$ has a nonzero coefficient in $c_1^3$. Theorem B then follows form
\begin{theorem}\label{ThmFixedLocusIsConstantCycle}
     The fixed locus $F$ is a contant cycle subvariety in $X$.
 \end{theorem}
 Indeed, Theorem~\ref{ThmFixedLocusIsConstantCycle} says that the nautral morphism $CH_0(F)_{hom}\to CH_0(X)_{hom}$ is zero.
 \begin{remark}
     Theorem~\ref{ThmFixedLocusIsConstantCycle} had been proved in~\cite{0CycleHK} in the case $r=1$.
 \end{remark}

\subsubsection{Proof of Theorem~\ref{ThmFixedLocusIsConstantCycle}}
Let $\mathcal P_{2,1}:=\{(x, l)\in X\times F_1(Y): l\subset P_x\}$ be the incidence variety. 

\begin{remark}\label{RmkI2AsSelfCorr}
    We have $I_2 = ^t\mathcal P_{2,1}\circ \mathcal P_{2,1}$ as self-correspondence of $X$.
\end{remark}

Let $x\in X=F_2(Y)$ be a general point and let $x'=\Psi(x)$ where $\Psi: X\dashrightarrow X$ is the Voisin map. In what follows, we use the following notation: for a plane $P$, the dual of $P$, defined as the set of lines in $P$, is denoted by $P^\vee$.
 \begin{proposition}\label{PropPx'-4PxConstant}
     The cycle $P_{x'}^{\vee} - 4 P_x^{\vee}\in CH_2(F_1(Y))_{\mathbb Q}$ does not depend on the choice of $x\in X$.
 \end{proposition}

Admitting Proposition~\ref{PropPx'-4PxConstant}, we conclude the proof of Theorem~\ref{ThmFixedLocusIsConstantCycle}.
\begin{proof}[Proof of Theorem~\ref{ThmFixedLocusIsConstantCycle}]
    By Remark~\ref{RmkI2AsSelfCorr}, we have 
    \[
    \mathcal P_{2,1}^*\circ \mathcal P_{2,1*} = I_{2*}: CH_0(X)\rightarrow CH_3(X).
    \]
     If $x\in F\subset X$, then in the statement of Proposition~\ref{PropPx'-4PxConstant}, $x' = x$ and thus $P_x^{\vee}\in CH_2(F_1(Y))_{\mathbb Q}$ is independent of the choice of $x\in F$. Hence, for any $x_1, x_2\in F$, we have $\mathcal P_{2,1*}(x_1-x_2) = P_{x_1}^{\vee} - P_{x_2}^{\vee} = 0 \in CH_2(F_1(Y))_{\mathbb Q}$. Therefore, $I_{2*}(x_1-x_2) = \mathcal P_{2,1}^*\mathcal P_{2,1*}(x_1 - x_2) = 0\in CH_0(X)_{\mathbb Q}$. Now, if in Equation (\ref{EqDecOfTheActionOfPsi})
     \[
     \Psi_*z = -8z + \gamma\cdot I_{2*}(z),
     \]
     we take $z = x_1 - x_2$, we get $z = -8z\in CH_0(X)_{\mathbb Q}$. Therefore, $z = 0 \in CH_0(X)$ as $CH_0(X)$ is torsion-free. This implies $x_1 = x_2 \in CH_0(X)$. Since $x_1, x_2\in F$ are arbitrarily chosen, we conclude that $F$ is a constant cycle subvariety.
\end{proof}

The rest of this section is devoted to the proof of Proposition~\ref{PropPx'-4PxConstant}. The proof relies on the geometry of cycles in a cubic fourfold $Y_4$ and its variety of lines $F_1(Y_4)$, which has been studied in~\cite{0CycleHK} and \cite{ChowHK}. the following relation is established in~\cite{0CycleHK} (see also \cite{ChowHK}).

\begin{theorem}[Voisin~\cite{0CycleHK}]\label{ThmVoisinQuadraticRelationOnDP}
    For a cubic fourfold $Y_4 \subset \mathbb{P}^5$ containing a plane $P$, let $P^\vee \subset F_1(Y_4)$ denote the variety of lines within $P$, and let $D_P \subset F_1(Y_4)$ represent the divisor comprising lines in $Y_4$ intersecting $P$. With $l \subset CH^1(F_1(Y_4))$ being the restriction of the Plücker line bundle class from $\mathrm{Gr}(2,6)$, there exist constants $\alpha, \beta \in \mathbb{Q}$, and $\gamma \in CH^2(F_1(Y_4))_{\mathbb{Q}}$, where $\alpha \neq 0$ and $\gamma$ is a restriction of a class $\delta \in CH^2(\mathrm{Gr}(2,6))_{\mathbb{Q}}$ that is independent of the chose of the plane $P$ and the cubic fourfold $Y_4$, such that:
    \[
    P^\vee = \alpha D_P^2 + \beta D_P \cdot l + \gamma
    \]
    within $CH^2(F_1(Y_4))_{\mathbb{Q}}$.
\end{theorem}

We will partially generalize this relation to the case where $Y_4$ has mild singularities.
\begin{corollary}\label{CorVoisinQuadraticRelationOnDP}
    Consider a cubic hypersurface $Y_4 \subset \mathbb{P}^5$ with at most simple double points as singularities, containing a plane $P$. Let $F_1(Y_4)_{sm}$ denote the smooth part of $F_1(Y_4)$. Define $P^\vee \subset F_1(Y_4)_{sm}$, $D_P \subset F_1(Y_4)_{sm}$, and $l \subset CH^1(F_1(Y_4)_{sm})$ as in Theorem~\ref{ThmVoisinQuadraticRelationOnDP}, but restricted to the smooth part of $F_1(Y_4)$. There exist constants $\alpha, \beta \in \mathbb{Q}$, and $\gamma \in CH^2(F_1(Y_4)_{sm})_{\mathbb{Q}}$, with $\alpha \neq 0$ and $\gamma$ as a restriction of a class $\delta \in CH^2(\mathrm{Gr}(2,6))_{\mathbb{Q}}$ that is independent of the chose of the plane $P$ and the cubic fourfold $Y_4$, such that:
    \begin{equation}\label{EqVoisinQuadraticRelationOnDP}
         P^\vee = \alpha D_P^2 + \beta D_P \cdot l + \gamma
    \end{equation}
    in $CH^2(F_1(Y_4)_{sm})_{\mathbb{Q}}$.
\end{corollary}

\begin{proof}
Let us consider the construction of the family $\mathcal{F}$ over $B$, where $B$ parametrizes pairs $(P, f)$ with $P$ being a plane in $\mathbb{P}^5$ and $f$ a cubic polynomial such that the hypersurface $Y_f$ defined by $f$ has at most simple double points as singularities, together with the condition that $f_{|P} = 0$. Let \[\mathcal F = \{
    ((P, f), l)\in B\times \mathrm{Gr}(2, 6): f|_l = 0\}
    \}.
    \]
In such a way, we make the Fano variety of lines $F_1(Y_f)$ into family over $B$.

 \begin{lemma}
        The family $p: \mathcal F\to B$ is flat.
    \end{lemma}
    \begin{proof}
        Each fiber of $p$ is a subvariety that is defined as the zero locus of the vector bundle $\mathrm{Sym}^3\mathcal E^*$ on $\mathrm{Gr}(2,6)$, with the expected dimension, where $\mathcal E$ is the tautological subbundle of $\mathrm{Gr}(2,6)$. By Koszul's resolution, each fiber has the same Hilbert polynomial. This implies that the family is flat. 
    \end{proof}
 Within the family $\mathcal F$, we have the following subvarieties. 
    $\mathcal P^\vee := \{ 
    ((P, f), l)\in B\times \mathrm{Gr}(2,6): l\subset P
    \}$ and $\mathcal D : = \{((P, f), l)\in \mathcal F: l\cap P\neq \emptyset\}$, representing lines within $P$ and lines intersecting $P$, respectively. Additionally, let $q: \mathcal F\to \mathrm{Gr}(2,6)$ be the second projection. Let $\mathcal L = q^*\mathcal O_{\mathrm{Gr}(2,6)}(1)$ be the pull-back of the Plücker line bundle. Let $\Gamma = q^*\delta\in CH^2(\mathcal F)_{\mathbb Q}$, where $\delta\in CH^2(\mathrm{Gr}(2,6))_{\mathbb Q}$ be the constant class as defined in Theorem~\ref{ThmVoisinQuadraticRelationOnDP}. Let $\alpha, \beta\in\mathbb Q$ be as in Theorem~\ref{ThmVoisinQuadraticRelationOnDP}.
    We consider the algebraic cycle $\mathcal Z = \mathcal P^{*} - \alpha \mathcal D^2 - \beta \mathcal D\cdot \mathcal L - \Gamma$. Theorem~\ref{ThmVoisinQuadraticRelationOnDP} implies that $\mathcal Z|_{\mathcal F_t} = 0 \in CH^2(\mathcal F_t)_{\mathbb Q}$ for $t\in B$ with smooth fibers. By the specialization of algebraic cycles~\cite[Proposition 1.4]{VoisinUnirational}, we conclude that $\mathcal Z|_{\mathcal F_t} = 0\in CH^2(\mathcal F_t)$ for all $t\in B$. For a singular fiber $\mathcal F_t$, we can restrict further to the smooth part of $\mathcal F_t$ and we get the desired result.
\end{proof}

\begin{remark}
    The reason we do not achieve the relation (\ref{EqVoisinQuadraticRelationOnDP}) for the whole of $F_1(Y_4)$ is that the divisor $D_P$ might encompass the singular locus of $F_1(Y_4)$, rendering it not a Cartier divisor, hence $D_P^2$ is not well-defined. However, upon restriction to the smooth part, all components are well-defined, and the restriction of $\mathcal{D}^2$ equates to $D_P^2$.
\end{remark}

We will also need the following

\begin{lemma}\label{Lmmj*DPIsConstant}
    Let $Y\subset \mathbb P^9$ be a general cubic eightfold and let $P\subset Y$ be a general plane contained in $Y$. Let $H_5$ be a general linear subspace of $\mathbb P^9$ containing $P$ such that $H_5\cap Y =: Y_4$ is a cubic hypersurface containing the plane $P$. Let $j: F_1(Y_4)\hookrightarrow F_1(Y)$ be the natural inclusion map. Let $D_P\subset F_1(Y_4)$ be as in Theorem~\ref{ThmVoisinQuadraticRelationOnDP}. Then the class $j_*D_P\in CH_3(F_1(Y))$ is independent of the choice of the plane $P\subset Y$ and of the linear subspace $H_5\subset \mathbb P^9$.
\end{lemma}
\begin{proof}
    Let $q: \mathcal P_{1,0}\to Y$ and $p: \mathcal P_{1,0}\to F_1(Y)$ be the universal {correspondence} of $Y$ and $F_1(Y)$.
    Let $\Sigma_P\subset F_1(Y)$ be the variety of lines in $Y$ that intersects the plane $P$. The class $\Sigma_P\in CH^5(F_1(Y))$ does not depend on the choice of $P\subset Y$, since $\Sigma_P = \mathcal P_{1,0}^*(P)$ and since $CH_2(Y) = \mathbb ZP$ by~\cite{Otwinowska}. Let $\tilde\Sigma_P\subset \mathcal P_{1,0}$ be the preimage of $P\subset Y$ via $q$. Then similarly, the Chow class of $\tilde\Sigma_P$ in $\mathcal P_{1,0}$ does not depend on the choice of $P$. Let us define two vector bundles $\mathcal E$ and $\mathcal H$ on $\mathcal P_{1,0}$ as follows. $\mathcal E$ is the pull-back of the universal subbundle over $F_1(Y)\subset \mathrm{Gr}(2,10)$ via $p$ and $\mathcal H$ is the pull-back of the Hopf bundle over $Y\subset \mathbb P^9$ via $q$. On $\tilde\Sigma_P$, we have a natural inclusion map $\mathcal H|_{\tilde\Sigma_P}\hookrightarrow \mathcal E|_{\tilde\Sigma_P}$ that induces a sujective morphism of vector bundles $\phi: \mathcal E|_{\tilde\Sigma_P}^*\to \mathcal H|_{\tilde\Sigma_P}^*$ that fits into the short exact sequence
    \begin{equation}\label{EqShortExactSequenceInCubicFourfold}
         0\to (\det\mathcal E^*\otimes \mathcal H)|_{\tilde\Sigma_P} \to \mathcal E|_{\tilde\Sigma_P}^*\to \mathcal H|_{\tilde\Sigma_P}^* \to 0.
    \end{equation}
   
    Over $\tilde\Sigma_P$, the defining equations of $H_5\subset \mathbb P^9$ induces a section $s$ of $\mathcal E^*|_{\tilde\Sigma_P}^{\oplus 4}$ that is zero when projected to $H^0(\tilde\Sigma_P, \mathcal H^*|_{\tilde\Sigma_P}^{\oplus 4})$. Hence, we can view $s$ as a section $\sigma_s$ of $(\det\mathcal E^*\otimes \mathcal H)|_{\tilde\Sigma_P}^{\oplus 4}$. Let $\tilde D_P$ be the zero locus of $\sigma_s$. Then $\tilde D_P$ parametrizes the pairs $(l, y)\in F_1(Y)\times Y$ such that $l\subset H_5$ and $l\cap P = y$. By the projection formula, we find that the Chow class of $\tilde D_P$ in $\mathcal P_{1,0}$ is $\tilde \Sigma_P\cdot c_1((\det\mathcal E^*\otimes \mathcal H))^4$, which is independent of the choice of $P$ and $H_5$. Since $D_P = p_*\tilde D_P$, the Chow class of $D_P$ in $F_1(Y)$ is independent of the choice of $P$ and $H_5$ as well.
\end{proof}

 For a plane $P\subset Y$, let $P^{\vee}\subset F_1(Y)$ be the subvariety of lines contained in $P$. Let $l\in CH^1(F_1(Y))$ be the restriction of the Plücker line bundle class of $\mathrm{Gr}(2, 10)$. Let $Y_4\subset Y$ be a linear section of $Y$ that has at most simple double points as singularities. Let $\Sigma\subset F_1(Y_4)$ be the singular locus of $F_1(Y_4)$ and let $j: F_1(Y_4)_{sm}\hookrightarrow F_1(Y)-\Sigma=: F_1(Y)^0$ be the inclusion map. Let $D_P$ be defined as in Corollary~\ref{CorVoisinQuadraticRelationOnDP}. Then Corollary~\ref{CorVoisinQuadraticRelationOnDP} and Lemma~\ref{Lmmj*DPIsConstant} imply the following
 \begin{corollary}\label{CorPIsConstantUpToDP2}
      In $CH_2(F_1(Y)^0)_{\mathbb Q}$, we have the following relation
     \[
     P^{\vee} = \alpha j_*(D_P^2) + c|_{F_1(Y)^0},
     \]
     where $\alpha\neq 0$ is a rational number and $c\in CH_2(F_1(Y))_{\mathbb Q}$ is a Chow class that is independent of the choice of the plane $P$.
 \end{corollary}

 \begin{proof}[Proof of Proposition~\ref{PropPx'-4PxConstant}]
     By the definition of the Voisin map, there is a unique linear subspace $H\subset \mathbb P^9$ of dimension $3$, such that $H\cap Y = 2 P_x + P_{x'}$. Let $H_5\subset \mathbb P^9$ be a linear subspace containing $H$ (thus containing both $P_x$ and $P_{x'}$), and let $Y_4 = Y\cap H_5$. 
     \begin{lemma}\label{LmmSingularityOfY4}
         For a general choice of $H_5$, the cubic hypersurface $Y_4$ has $4$ simple double points as the only singularities.
     \end{lemma}
     \begin{proof}
         The base locus of the linear system $L = \{H_5'\cap Y: H\subset H_5'\subset \mathbb P^9\}$ is $P_x\cup P_{x'}$. Hence, by Bertini's theorem, for a general $H_5$, the cubic hypersurface $Y_4 := Y\cap H_5$ is smooth outside $P_x\cup P_{x'}$. Let us write $H_5 = \{(x_0: x_1: \ldots: x_5)\}$ and we can assume that $H = \{(x_0: x_1: x_2: x_3: 0: 0)\}$, $P_x = \{(x_0: x_1: x_2: 0: 0: 0)\}$ and $P_{x'} = \{(0: x_1: x_2: x_3: 0: 0)\}$. The fact that $H\cap Y_4 = 2 P_x + P_{x'}$ means that the defining equation of $Y_4\subset H_5$ can be written as
         \[
         f(x_0,\ldots, x_5) = x_3^2x_0 + x_4Q_1(x_0, \ldots, x_5) + x_5Q_2(x_0, \ldots, x_5),
         \]
         where $Q_1$ and $Q_2$ are quadratic polynomials. Let $y\in Y_4$ be a singular point. Since $y\in H$, one must have $x_4(y) = x_5(y) = 0$. Writing $f_{x_i} = \frac{\partial f}{\partial x_i}$, we find by direct calculations that $f_{x_0}(y) = x_3(y)^2$,  $f_{x_1}(y) = f_{x_2}(y) = 0 $, $f_{x_3}(y) = 2x_0(y)x_3(y)$, $f_{x_4}(y) = Q_1(y)$ and $f_{x_5}(y) = Q_2(y)$. The fact $y\in Y_4$ is singular implies that $f_{x_i}(y) = 0$ for any $x_i$. Taken together, we find that $y\in H_5$ satisfies the following system of equations
         \[
         \left\{\begin{array}{cc}
              x_4 &= 0   \\
               x_5 &= 0 \\
               x_3 &= 0 \\
               Q_1(x_0, \ldots, x_5) &= 0 \\
               Q_2(x_0, \ldots, x_5) &= 0
         \end{array}\right.
         \]
         By the generality of $Y\subset \mathbb P^9$ and $H_5$, the solutions of this system of equations are four points in $H_5$. To show that the four singular points are simple double points. We do a local check. Let $y$ be one of the singular points. Up to a change of coordinates of $P_x$, we may assume without loss of generality that $x_0(y) \neq 0$ and $x_1(y) = x_2(y) =0$, namely, $y = (1: 0: 0: 0: 0: 0)\in H_5$. On the open affine subset $U_0$ of $H_5$ defined by $x_0 = 1$, the cubic hypersurface $Y_4\cap U_0$ is defined by the equation $x_3^2 + x_4Q_1(1, x_1, \ldots, x_5) + x_5Q_2(1, x_1, \ldots, x_5)$. The Taylor expansion of this polynomial around the point $y = (0, 0, 0,0,0)\in \mathbb A^5\cong U_0$ is the polynomial itself. The fact that this polynomial does not have degree $1$ term is simply because $y$ is a singular point. To show that $y$ is a double point, we only need to make sure that the degree $2$ term of the expression $x_3^2 + x_4Q_1(1, x_1, \ldots, x_5) + x_5Q_2(1, x_1, \ldots, x_5)$ is non-degenerate, and this condition is clearly satisfied for general $Q_1$ and $Q_2$.
     \end{proof}
     Let $y_1, y_2, y_3, y_4$ be the four singular points of $Y_4$. Let $\Sigma\subset F_1(Y_4)$ be the singular locus of $F_1(Y_4)$. Then $\Sigma$ is the union of four surfaces $\Sigma_1$, $\Sigma_2$, $\Sigma_3$ and $\Sigma_4$, parametrizing the lines in $Y_4$ that pass through the point $y_1, y_2, y_3$ and $y_4$, respectively. Let $F_1(Y)^0= F_1(Y)-\Sigma$
     Let $j: F_1(Y_4)_{sm}\hookrightarrow F_1(Y)^0$ be the natural inclusion. By Corollary~\ref{CorPIsConstantUpToDP2}, we have \begin{equation}\label{EqPxDPx}
         P_x^{*} = \alpha j_*(D_{P_x}^2) + c'|_{F_1(Y)^0}\in CH_2(F_1(Y))_{\mathbb Q}
     \end{equation}
     and 
     \begin{equation}\label{EqPx'DPx'}
         P_{x'}^{*} = \alpha j_*(D_{P_{x'}}^2) + c'|_{F_1(Y)^0}\in CH_2(F_1(Y)^0)_{\mathbb Q}
     \end{equation}
     for some constant $c'\in CH_2(F_1(Y))_{\mathbb Q}$. Now let $\mathcal P = \{(l, y)\in F_1(Y_4)_{sm}\times Y_4: y\in l\}$ be the {incidence} {correspondence}, then for any plane $P\subset Y_4$, we have $D_P = \mathcal P^*(P)$. It is clear that $2P_x + P_{x'} = h^2$ in $CH_2(Y_4)$, so $2D_{P_x} + D_{P_{x'}} = \mathcal P^*(h^2) = l\in CH_3(F_1(Y_4)_{sm})$. Taken into account of the equations (\ref{EqPxDPx}) and (\ref{EqPx'DPx'}), we find that 
     \[
     P_{x'}^{*} = 4 P_x^{*} - 4\alpha j_*D_{P_x}\cdot l + \alpha l^2 - 3 c' \in CH_2(F_1(Y)^0)_{\mathbb Q}.
     \]
     Taking into account of Lemma~\ref{Lmmj*DPIsConstant}, $P_{x'}^{*} - 4 P_x^{*} = c|_{F_1(Y)^0}\in CH_2(F_1(Y)^0)_{\mathbb Q}$, where $c\in CH_2(F_1(Y))_{\mathbb Q}$ is a constant $2$-cycle on $F_1(Y)$. Hence, in $F_1(Y)$, the cycle $P_{x'}^{*} - 4 P_x^{*} - c$ is supported on $\Sigma$, the singular locus of $F_1(Y_4)$. Thus, we need to understand the geometry and Chow classes of $\Sigma$. 

     Following Lemma~\ref{LmmSingularityOfY4}, let $y_1, y_2, y_3, y_4$ be the four singular points of $Y_4$. Then $\Sigma$ is the union of four surfaces $\Sigma_1$, $\Sigma_2$, $\Sigma_3$ and $\Sigma_4$, parametrizing the lines in $Y_4$ that pass through the point $y_1, y_2, y_3$ and $y_4$, respectively. We have the following two lemmas about the geometry of the surfaces $\Sigma_i$ that we will prove later. 
     \begin{lemma}\label{LmmIrreducibilityOfTheFourSurfaces}
         The surfaces $\Sigma_1$, $\Sigma_2$, $\Sigma_3$ and $\Sigma_4$ are irreducible. In particular, $CH_2(\Sigma) = \oplus_{i=1}^4 CH_2(\Sigma_i)$.
     \end{lemma}
     \begin{lemma}\label{LmmChowClassOfTheFourSurfaces}
         Let $H_5\subset \mathbb P^9$ be a general linear subspace of dimension $5$ such that the linear section $Y_4:= Y\cap H_5$ has only simple double singularities and let $y\in Y_4$ be a singular point. Let $S$ be the surface of lines in $Y_4$ that pass through the point $y$. Then the Chow class of $S$ in $F_1(Y)$ does not depend on the choice of $H_5$ and $y$.
     \end{lemma}
     By Lemma~\ref{LmmIrreducibilityOfTheFourSurfaces} and the fact that $P_{x'} - 4P_x - c$ is supported on $\Sigma$, we conclude that $P_{x'} - 4P_x - c = \sum_{i=1}^4 a_i\Sigma_i$, with $a_i\in\mathbb Q$. The cohomological class of $P_{x'} - 4P_x - c$ is clearly a constant, thus the cohomological class of $\sum_{i=1}^4 a_i\Sigma_i$ is constant. Lemma~\ref{LmmChowClassOfTheFourSurfaces} then implies that the Chow class of $\sum_{i=1}^4 a_i\Sigma_i$ is also constant. Therefore, $P_{x'} - P_x = c + \sum_{i = 1}^4 a_i\Sigma_i$ is a constant in $CH_2(F_1(Y))_{\mathbb Q}$, as desired. This concludes the proof of Proposition~\ref{PropPx'-4PxConstant}.
 \end{proof}
 \begin{proof}[Proof of Lemma~\ref{LmmIrreducibilityOfTheFourSurfaces}]
         To check the irreducibility, we return to the proof of Lemma~\ref{LmmSingularityOfY4}. Using the notation there, let 
         \[g(x_1, \ldots, x_5) = f(0, x_1, \ldots, x_5) =  x_4Q_1(0, x_1, \ldots, x_5) + x_5 Q_2(0, x_1, 
         \ldots, x_5)\]
         and let $q(x_1, \ldots, x_5)$ be the degree $2$ part of the polynomial $x_3^2 + x_4Q_1(1, x_1, \ldots, x_5) + x_5Q_2(1, x_1, \ldots, x_5)$. Then
         the surface $\Sigma_1$ of lines passing through the singular point $y = (1: 0: 0: 0: 0: 0)$ is the subvariety in $\mathbb P^4 = \{(x_1: \ldots: x_5)\}$ cut by the  equations $g(x_1, \ldots, x_5) = 0$ and $q(x_1, \ldots, x_5) = 0$. From the expression, we see that $g(x_1, \ldots, x_5)$ depends on the coefficients of the terms $x_ix_j$ with $1\leq i, j\leq 5$ in $Q(x_0, x_1, \ldots, x_5)$ whereas $q(x_1, \ldots, x_5)$  depends on the coefficients of the terms $x_0x_i$ for $1\leq i\leq 5$. So the choice of coefficients of $g(x_1, \ldots, x_5)$ and $q(x_1, \ldots, x_5)$ does not have influence on each other. Now we fix one choice of smooth $Q=\{q=0\}$. Then varying $C = \{g = 0\}$, the base points of $C\cap Q$ is the line $(s: t: 0: 0: 0)$ corresponding to the lines on the plane $P_x$ passing through the point $y$. Hence, for a general choice of $q$ and $g$, the surface $\Sigma_1 = C\cap Q$ is smooth outside the line $(s: t: 0: 0: 0)$. Writing $Q_1(x_0, \ldots, x_5) = \sum_{i, j = 0}^5 a_{ij}x_ix_j$ and $Q_2(x_0, \ldots, x_5) = \sum_{i, j = 0}^5 b_{ij}x_ix_j$, the Jacobian matrix of the polynomials $q(x_1, \ldots, x_5)$ and $g(x_1, \ldots, x_5)$ at the point $(s: t: 0: 0: 0)$ is
    \[
    \begin{pmatrix}
        0 & 0 & 0 &  a_{01} s + a_{02} t &  b_{01} s + b_{02} t\\
        0 & 0 & 0 & 2 a_{12} st & 2 b_{12} st 
    \end{pmatrix}.
    \]
    This matrix does not have full rank only if $s = 0$ or $t = 0$ or $(a_{01} - b_{01}) s + (a_{02} - b_{02}) t = 0$, corresponding to the three lines passing through the other three singular points. Therefore, the singular locus of $\Sigma_1$ is of codimension $2$. But a reducible complete intersection of dimension $\geq 1$ in the projective space always have codimension $1$ singular locus by the Fulton–Hansen connectedness theorem~\cite{Connectedness}. Hence, $\Sigma_1$ is irreducible, as desired.
    \end{proof}

\begin{proof}[Proof of Lemma~\ref{LmmChowClassOfTheFourSurfaces}]
         Let $H$ be the projective tangent space of $Y$ at point $y$. It is a linear subspace of dimension $8$ in $\mathbb P^9$. The linear section $Y_4$ being singular at $y$ is equivalent to the relation $H_5\subset H$. Let $F_{1, y}(Y)$ be the variety of lines in $Y$ that passes through $y$. Let $\mathcal P = \{(l, y)\in F_{1}(Y)\times Y: y\in l\}$ be the {incidence} variety. Then $F_{1,y}(Y)$ can be identified with a subvariety of $\mathcal P$ given by the preimage of the point $y$ under the projection map $q: \mathcal P\to Y$.
         Since every line in $Y$ passing through $y$ is contained in $H$, the subvariety $S$ of $F_{1, y}(Y)$ parametrizing the lines that is furthermore included in $H_5$ is given by the zero locus of a section $\sigma$ of the vector bundle $(\mathcal E^*)^{\oplus 3}$. However, as the point $y$ already lies in $H_5$, if we restrict $\sigma$ on $y$ via the following morphism
         $(\mathcal E^*)^{\oplus 3} \to \mathcal O^{\oplus 3}$,
         we get zero. Thus, $S$ can be viewed as the zero locus of a section of $(\det(\mathcal E^*))^3$, with expected dimension. Therefore, the Chow class of $S$ in $F_1(Y)$ is given by $c_1(\mathcal E^*)^3\cdot F_{1, y}(Y)$. Note that $F_{1, y}(Y) = \mathcal P^*(y) \in CH(F_1(Y))$ with $\mathcal P$ the canonical {correspondence} between $F_1(Y)$ and $Y$. Since $CH_0(Y)$ is trivial, the class $F_{1, y}(Y)\in CH(F_1(Y))$ is independent of $y$. In conclusion, the constant Chow class of $S$ in $F_1(Y)$ is independent of the choice of $H_5$ and $y$. 
     \end{proof}

\subsubsection{Chow Class of $F$ in $X$}\label{SectionChowClassOfFInX}
In this section, we determine the Chow class of $F$ within $X$. Consider $\mathcal{F}$ as the tautological subbundle of $\mathrm{Gr}(4, 10)$ (resp. $\mathrm{Gr}(r+2, n+1)$ in the general case), and $\mathcal{E}$ as the tautological subbundle of $X = F_2(Y)$ (resp. $X = F_r(Y)$ in the general case). Let $c_i$ represent $c_i(\mathcal{E}^*)$ over $X$. In the case $r = 2$, we have

\begin{proposition}\label{ThmChowClassOfF}
The Chow class of $F$ within $X$ is expressed as $-20 c_1^3 + 110c_1c_2 + 49 c_3$ in $CH^3(X)$.
\end{proposition}
\begin{proof}
    Let us start by giving the general method for any $r$ and then specialize to $r=2$ for explicit calculations. Consider the following stratification of $X\times \mathrm{Gr}(r+2,n+1)$:
    \[
    \begin{tikzcd}
        X\times \mathrm{Gr}(r+2, n+1)\arrow[d, hookleftarrow, "\gamma"]\\
        \tilde{\tilde X}\arrow[d, hookleftarrow, "\beta"] :=\{(x, P_{r+1})\in X\times \mathrm{Gr}(r+2,n+1)| P_{r+1}\cap Y\supset P_x\}\\
        \tilde X \arrow[d, hookleftarrow,"\alpha"] :=\{(x,P_{r+1})\in X\times \mathrm{Gr}(r+2, n+1)| P_{r+1}\cap Y\supset 2P_x\}\\
        \tilde F:=\{(x, P_{r+1})\in X\times \mathrm{Gr}(r+2,n+1)|P_{r+1}\cap Y\supset 3P_x\}
    \end{tikzcd}
    \]
    Let $pr_1: X\times \mathrm{Gr}(r+2,n+1)\to X$ and $pr_2: X\times \mathrm{Gr}(r+2, n+1)\to \mathrm{Gr}(r+2, n+1)$ be the projection maps. It is not hard to see that $F=(pr_1\circ\gamma\circ\beta\circ\alpha)(\tilde F)$. We will see shortly that the Chow class of $\tilde F$ in $X\times \mathrm{Gr}(r+2, n+1)$ lies in a subring generated by the Chern classes of $pr_1^*\mathcal E^*$ and $pr_2^*\mathcal F^*$, say 
    \begin{equation}\label{EqExpOfTildeF}
        [\tilde F]=\sum_i pr_1^*P_i(c_k(\mathcal E^*))\cdot pr_2^*Q_i(c_l(\mathcal F^*))
    \end{equation}
    in $CH^{r+1+(n-r-1)(r+2)}(X\times \mathrm{Gr}(r+2, n+1))$, where $P_i$ and $Q_i$ are polynomials with adequate degrees. Then the class of $F$ in $X$ is $pr_{1*}([\tilde F])=\sum_iP_i(c_k(\mathcal E^*))\cdot pr_{1*}pr_2^*Q_i(c_l(\mathcal F^*))$. In this expression, 
    $pr_{1*}pr_2^*Q_i(c_l(\mathcal F^*))\neq 0$ only if the weighted degree of $Q_i(c_l(\mathcal F^*))$ is $(n-r-1)(r+2)$, namely the dimension of $\mathrm{Gr}(r+2, n+1)$ (or equivalently, the weignted degree of $P_i(c_k(\mathcal E^*))$ is $r+1$). In this case, in $CH_0(\mathrm{Gr}(r+2,n+1))$, the class $Q_i(c_l(\mathcal F^*))$ is a multiple of a point $o\in \mathrm{Gr}(r+2, n+1)$. The coefficient, denoted by $q_i$, is calculable using Schubert calculus, once we know the expression of $Q_i(c_l(\mathcal F^*))$. Taken together, the class of $F$ in $CH^{r+1}(X)$ is given by
    \[
    [F]=\sum_{i: \deg P_i=r+1} q_i P_i(c_k(\mathcal E^*)).
    \]
    Now let us prove (\ref{EqExpOfTildeF}), i.e., $[\tilde F]$ is generated by $pr_1^*\mathcal E^*$ and $pr_2^*\mathcal F^*$. The subvariety $\tilde{\tilde X}$ of $X\times \mathrm{Gr}(r+2, n+1)$ is defined by the condition $P_x\subset P_{r+1}$. Hence, we can view $\tilde{\tilde X}$ as the locus where the composite map of vector bundles $pr_1^*\mathcal E\to V_{n+1}\to V_{n+1}/pr_2^*\mathcal F $ is zero, where $V_{n+1}$ is the trivial bundle of rank $n+1$ over $X\times \mathrm{Gr}(r+2, n+1)$. Hence, $\tilde{\tilde X}$ is the zero locus of a section of $pr_1^*\mathcal E^*\otimes V_{n+1}/pr_2^*\mathcal F$. Now over $\tilde{\tilde X}$, we have a natural injection $pr_1^*\mathcal E|_{\tilde{\tilde X}}\to pr_2^*\mathcal F|_{\tilde{\tilde X}}$. The defining polynomial $f$ of $Y$ induces a section of $pr_2^*\mathrm{Sym}^3\mathcal F^*|_{\tilde{\tilde X}}$. Since $P_x\subset Y$ for any $x\in X$, this section vanishes when passing to $pr_1^*\mathrm{Sym}^3\mathcal E^*|_{\tilde{\tilde X}}$. Hence, the section is actually a section $\tilde{\tilde f}$ of $pr_2^*\mathrm{Sym}^2\mathcal F^*|_{\tilde{\tilde X}}\otimes (pr_2^*\mathcal F|_{\tilde{\tilde X}}/pr_1^*\mathcal E|_{\tilde{\tilde X}})^*$. The subvariety $\tilde X$ is  the locus of $\tilde{\tilde X}$ where, when passing to $pr_1^*\mathrm{Sym}^2\mathcal E^*|_{\tilde{\tilde X}}\otimes (pr_2^*\mathcal F|_{\tilde{\tilde X}}/pr_1^*\mathcal E|_{\tilde{\tilde X}})^*$, the section $\bar{\tilde{\tilde f}}$ is zero. Hence, $\tilde X$ is the zero locus of a section of the vector bundle $pr_1^*\mathrm{Sym}^2\mathcal E^*|_{\tilde{\tilde X}}\otimes (pr_2^*\mathcal F|_{\tilde{\tilde X}}/pr_1^*\mathcal E|_{\tilde{\tilde X}})^*$. Similar argument shows that $\tilde F$ is the zero locus of a section of the vector bundle $pr_1^*\mathcal E^*|_{\tilde X}\otimes (pr_2^*\mathcal F|_{\tilde X}/pr_1^*\mathcal E|_{\tilde X})^{*\otimes 2}$. Taken together, the class of $\tilde F$ in $CH(X\times \mathrm{Gr}(r+2,n+1))$ is given by the following class
    \[
    e(pr_1^*\mathcal E^*\otimes (V_{n+1}-pr_2^*\mathcal F))\cdot e((pr_2^*\mathcal F^*-pr_1^*\mathcal E^*)\otimes pr_1^*\mathrm{Sym}^2\mathcal E^*)\cdot e((pr_2^*\mathcal F^*-pr_1^*\mathcal E^*)^{\otimes 2}\otimes pr_1^*\mathcal E^*)
    \]
    where $e(...)$ is the Euler class and we have written the vector bundles as their class in K groups to avoid any confusions. This gives an expression of the class $\tilde F$ as generated by the Chern classes of $pr_1^*\mathcal E^*$ and $pr_2^*\mathcal F^*$.
 The explicit calculations in the case $r=2$ is given in Appendix~\ref{SectionCalculations} using Macaulay2~\cite{Macaulay2} (and its Schubert2 package), and is also verified using IntersectionTheory package~\cite{IntersectionTheory} (which uses the Singular computer algebra system~\cite{Singular}). 
\end{proof}
\begin{remark}
    Using the same argument (see~\cite[Lemma 5.2]{FixedLocusII}), one finds that for the Fano variety of lines of cubic fourfold, the fixed locus of the Voisin map is $21c_2$—a result that coincides with the result of~\cite[Theorem A]{FixedLocus}.
\end{remark}

\section{Proof of Theorem C}\label{SectionProofOfC}

Let us analysis the action of the Voisin map $\Psi$ on the Chern classes of $X$ and on the divisor class of $X$ respectively. Note that for $r\geq 2$, $CH^1(X)$ is generated by only one divisor $h$, which is the restriction of Plücker line bundle on $\mathrm{Gr}(r+1, n+1)$.
\begin{lemma}\label{LmmActionOfPsiOnDivisor}
    For any $r\geq 1$, we have $\Psi^*h=(3r+4)h$.
\end{lemma}
\begin{proof}
    It is not hard to see that the codimension of $\mathrm{Ind}_0\subset X$ is $2$. Hence, the result will not change if we replace $X$ by $X':=X-\mathrm{Ind}_0$. We use the same notations for the $\mathcal O_X$-modules restricted to $X'$. Then $\mathcal F$ becomes a vector bundle of rank $r+2$ on $X'$ which fits in the following short exact sequence
    \begin{equation}\label{EqShortExactDefF}
        0\to \mathcal F\to V_{n+1}\otimes \mathcal O_X\to \mathrm{Sym}^2\mathcal E^*\to 0.
    \end{equation}
    From this short exact sequence, we deduce that 
    \begin{equation}
        c_1(\mathcal F)=-(r+2)h.
    \end{equation}
    The homogeneous polynomial $f$ induces a section $\sigma_f$ of the vector bundle $\mathrm{Sym}^3\mathcal F^*$, and the fact that $\mathbb P(\mathcal F|_x)$ is the $(r+1)$-space that is tangent to $Y$ along $P_x$ shows that $\sigma_f$ can be viewed as an element in $H^0(X', (\mathcal F/\mathcal E)^*\otimes (\mathcal F/\mathcal E)^*\otimes \mathcal F^*)$ via the following short exact sequences
    \[
    0\to (\mathcal F/\mathcal E)^*\otimes \mathrm{Sym}^2\mathcal F^*\to \mathrm{Sym}^3\mathcal F^*\to \mathrm{Sym}^3E^*\to 0,
    \]
    \[
    0\to (\mathcal F/\mathcal E)^*\otimes (\mathcal F/\mathcal E)^*\otimes \mathcal F^*\to (\mathcal F/\mathcal E)^*\otimes \mathrm{Sym}^2\mathcal F^*\to (\mathcal F/\mathcal E)^*\otimes \mathrm{Sym}^2\mathcal E^*\to 0.
    \]
    Therefore, $\sigma_f$ induces a morphism $\phi: \mathcal F\to (\mathcal F/\mathcal E)^*\otimes (\mathcal F/\mathcal E)^*$. Then the locus of $X$ where $\phi$ is not surjective is $\mathrm{Ind}_1$. It is not hard to see that the codimension of $\mathrm{Ind}_1$ in $X$ is $r+2$ if $r\geq 2$ and $\mathrm{Ind}_1 = \emptyset$ if $r = 1$.
    Therefore, the result will not change if we replace $X'$ by $X''=X-\mathrm{Ind}_0-\mathrm{Ind}_1$ and we use the same notations for the restrictions of $\mathcal O_X$-modules to $X''$. On $X''$, the vector bundle $\Psi^*\mathcal E$ is exactly the kernel of $\phi$ and $\Psi^*\mathcal E$ fits into the following short exact sequence
    \[
    0\to \Psi^*\mathcal E\to \mathcal F\to (\mathcal F/\mathcal E)^*\otimes (\mathcal F/\mathcal E)^*\to 0.
    \]
    Therefore, $\Psi^*c_1(\mathcal E)=c_1(\Psi^*\mathcal E)=c_1(\mathcal F)-2c_1((\mathcal F/\mathcal E)^*)=-(r+2)h-2(-(-(r+2)h+h))=-(3r+4)h$. Since $c_1(\mathcal E)=-h$, we find that $\Psi^*h=(3r+4)h$, as desired.
\end{proof}

\begin{remark}
    When $r=1$ and $n=5$, $Y$ is a cubic fourfold and $X$ is a hyper-Kähler fourfold. Our result recovers the result of~\cite{Amerik} and~\cite[Proposition 21.4]{ChowHK} which states $\Psi^*h=7h$.
\end{remark}

\begin{lemma}\label{LmmActionOfPsiOnTangent}
    Let $X^0 = X-\mathrm{Ind}$ be the locus where $X$ is defined. We have $(\Psi^*T_X)|_{X^0} = T_{X^0}$.
\end{lemma}

\begin{proof}
    This result has been implicitly proven in~\cite{KCorr}. Let $\tau: \tilde X\to X$, $\tilde\Psi: \tilde X\to X$ be the desingularisation of the indeterminacies of the Voisin map $\Psi: X\dashrightarrow X$. Since $X$ is a $K$-trivial variety, the exceptional divisor of $\tau$ coincides with the ramification locus of $\tilde\Psi$ (see~\cite[Lemma 4]{KCorr}). Therefore, the Voisin map restricted to the defined domain $\Psi|_{X^0}: X^0\to X$ is étale, which implies that the relative tangent bundle $T_{X^0/X}$ by the map $\Psi|_{X^0}$ is zero. Hence the result. 
\end{proof} 

\begin{proof}[Proof of Theorem C]
    We only consider the case $r=2$. Let $M\in CH_0(X)_{hom}$ be a class generated by $h$ and by $c_i(X)$. Let us assume $M$ is a monomial of the form $h^k \prod c_i(X)$. By Lemma~\ref{LmmActionOfPsiOnDivisor} and Lemma~\ref{LmmActionOfPsiOnTangent}, we have
    \begin{equation}\label{EqPsiOnM1}
        (\Psi^*M)|_{X^0} = 10^k\cdot M|_{X^0},
    \end{equation} 
    where $X^0 = X-\mathrm{Ind}$ is the locus where $\Psi$ is defined. On the other hand, by Theorem B, we have \begin{equation}\label{EqPsiOnM2}
        \Psi^*M = -8 M.
    \end{equation} 
    By (\ref{EqPsiOnM1}) and (\ref{EqPsiOnM2}), we find that $M|_{X^0} = 0$, which means that $M$ is supported on $\mathrm{Ind}$, as desired.
\end{proof}

\appendix
\section{Detailed Calculations for Proposition~\ref{ThmChowClassOfF}}\label{SectionCalculations}
The following two code snippets calculate the Chow class of $F\subset X$. 

The first snippet, utilizing Macaulay2~\cite{Macaulay2} with the Schubert2 package, is as follows. Similar calculations can also be found in~\cite[Lemma 5.2]{FixedLocusII}.

\lstset{
  basicstyle=\ttfamily\small, 
  breaklines=true,            
}

\begin{lstlisting}
    i1 : needsPackage "Schubert2";
    
    i2 : G = flagBundle({3}, 10); E = dual (bundles G)_1;
    
    i4 : GxG = flagBundle({4}, 10*OO_G); F = dual (bundles GxG)_1;
    
    i6 : pE = E*OO_GxG; pF = F;
    
    i8 : (GxG/G)_* (ctop(dual pE * (10*OO_GxG-pF)) * ctop((dual pF - dual pE)*symmetricPower_2 dual pE) * ctop((dual pF-dual pE)*(dual pF-dual pE)*dual pE))
    
              3
    o8 = - 20H    + 110H   H    + 49H
              1,1       1,1 1,2      1,3
    
    i9 : chern dual E
    
    o9 = 1 + H    + H    + H
              1,1    1,2    1,3
\end{lstlisting}

The second snippet, using the IntersectionTheory package~\cite{IntersectionTheory}, is provided below:
\lstset{
  basicstyle=\ttfamily\small,  
  breaklines=true,             
  mathescape=true              
}

\begin{lstlisting}
    julia> using IntersectionTheory
    
    julia> r = 2; n = binomial(r+3,2)-1;
    
    julia> Gr1 = grassmannian(r+1,n+1); E = bundles(Gr1)[1]; Gr2 = grassmannian(r+2,n+1); F = bundles(Gr2)[1];
    
    julia> GxG = Gr1*Gr2; pE = E*OO(GxG); pF = F*OO(GxG);
    
    julia> pushforward(GxG -> Gr1, ctop(dual(pE) * ((n+1) * OO(GxG) - pF)) * ctop((dual(pF) - dual(pE)) * symmetric_power(2, dual(pE))) * ctop((dual(pF) - dual(pE))^2 * dual(pE)))
    # Output: 20*$c_1^3$ - 110*$c_1c_2$ - 49*$c_3$
    
    julia> chern(dual(E))
    # Output: 1 - $c_1$ + $c_2$ - $c_3$
\end{lstlisting}
In the second snippet, for simplicity in coding, the symbol $c_i$ denotes the Chern classe of $\mathcal{E}$ rather than $\mathcal{E}^*$, accounting for a sign difference.

Both snippets confirm the results as stated in Proposition~\ref{ThmChowClassOfF}.

\end{document}